\documentclass[10pt]{amsart}
\usepackage{amsthm}

\usepackage{amsmath, amssymb}
\usepackage{pst-all}
\usepackage{rotating,xspace}

\theoremstyle{plain}
\newtheorem{prop}{Proposition}[section]
\newtheorem{coro}[prop]{Corollary}
\newtheorem{fact}[prop]{Fact}
\newtheorem{lemm}[prop]{Lemma}
\newtheorem{ques}[prop]{Question}
\theoremstyle{definition}
\newtheorem{defi}[prop]{Definition}
\newtheorem{exam}[prop]{Example}
\newtheorem{rema}[prop]{Remark}

\newenvironment{prproof}
{\begin{proof}[Proof (sketch)]}
{\end{proof}}

\numberwithin{equation}{section}

\def\Reff#1; #2; #3; #4; #5; #6; #7\par{%
\bibitem{#1} #2, {\it #3}, #4 {\bf #5} (#6) #7}
\def\Ref#1; #2; #3; #4\par{%
\bibitem{#1} #2, {\it #3}, #4}

\psset{unit=1mm}
\psset{arrowsize=6pt}

\renewcommand\aa{a}
\renewcommand\AA{A}
\newcommand\aaa[1]{a_{#1}}
\newcommand\act{\mathbin{\scriptscriptstyle\bullet}}
\newcommand\Aut{\mathsf{Aut}}

\newcommand\bb{b}
\newcommand\BB{B}
\newcommand\Bi{B_\infty}
\newcommand\Bii{B_\infty^{\mathsf{sm}}}
\newcommand\Bisp{B_\infty^{\mathsf{sp}}}
\newcommand\BP[1]{B^{\scriptscriptstyle+}_{#1}}
\newcommand\BPi{\BP\infty}
\newcommand\BR[1]{B_{#1}}
\newcommand\br{\beta}
\newcommand\brr{\gamma}
\newcommand\BVhat{\widehat{B\HS{-0.2}V}}

\newcommand\CBi{\CBR\infty}
\newcommand\CBR[1]{C\HS{-0.4}B_{#1}}
\newcommand\cc{c}
\newcommand\cl[1]{\mathsf{cl}(#1)}
\newcommand\coIm{\mathsf{coIm}}
\newcommand\comp{\mathbin{\scriptscriptstyle\circ}}
\newcommand\compl[1]{c(#1)}
\newcommand\concat{\mathbin{{}^{\frown}}}

\newcommand\dist{d}

\newcommand\EBi{E\HS{-0.4}B_\infty}
\newcommand\ee{e}
\newcommand\eqLD{\mathrel{=_{\scriptscriptstyle\mathsf{LD}}}\nobreak}

\newcommand\ff{f}

\newcommand\Fi{F_\infty}

\renewcommand\ge{\geqslant\nobreak}
\renewcommand\gg{g}
\newcommand\GG{G}
\newcommand\GLD{G_{\LD}}

\newcommand\hh{h}
\newcommand\HS[1]{\leavevmode\null\hspace{#1mm}}

\newcommand\id{\mathsf{id}}
\newcommand\ii{i}
\newcommand\Ii{\mathfrak{I}_\infty}
\newcommand\II{I}
\newcommand\Iisp{\mathfrak{I}_\infty^{\mathsf{sp}}}
\renewcommand\Im{\mathsf{Im}}
\newcommand\inv{^{-1}}
\newcounter{ITEM}
\newcommand\ITEM[1]{\setcounter{ITEM}{#1}\leavevmode\hbox{\rm(\roman{ITEM})}}
\newcommand\Iter{\mathsf{Iter}}

\newcommand\jj{j}

\newcommand\Ker{\mathsf{Ker}}
\newcommand\kk{k}

\newcommand\LD{\mathrm{LD}}
\newcommand\LDop[1]{\hbox{\scshape{ld}}_{#1}}
\renewcommand\le{\leqslant\nobreak}
\newcommand\LL{L}

\newcommand\mm{m}
\newcommand\mults{\mathbin{\scriptstyle\sqsupset^*}\nobreak}

\newcommand\nn{n}
\newcommand\NN{N}
\newcommand\NNNN{\mathbb{N}}

\newcommand\op{\mathbin{\triangleright}\nobreak}

\newcommand\opb{\mathbin{\overline{\op}}}
\newcommand\opG{\mathbin{\ast}}
\newcommand\oph{\mathrel{\widehat\op}}
\newcommand\opR{\mathbin{\triangleleft}}
\newcommand\opt{\mathbin{\widetilde\op}}

\newcommand\PBi{B_\bullet}
\newcommand\PBisp{\PBi^{\mathsf{sp}}}
\newcommand\pdots{\mathrel{\HS{0.2}{\cdot}{\cdot}{\cdot}\HS{0.2}}}
\newcommand\perm{\mathsf{perm}}
\newcommand\pp{p}
\newcommand\pref{\mathbin{\scriptstyle\sqsubset}}
\newcommand\prefs{\mathbin{\scriptstyle\sqsubset^*}}

\newcommand\qq{q}

\newcommand\RD{\mathrm{RD}}
\newcommand\Rel[1]{\mathsf{Rel}_{#1}}

\newcommand\RR{R}

\newcommand\sh{\mathsf{sh}}
\newcommand\SH{\hbox{\scshape{sh}}}
\newcommand\shprod{\textstyle\prod\nolimits^{\mathsf{sh}}\HS{-0.5}}
\newcommand\SHt{\widetilde\SH}
\newcommand\shtr{\mathsf{shtr}}
\newcommand\Si{\mathfrak{S}_\infty}
\newcommand\sig[1]{\sigma_{#1}}
\newcommand\sigg[2]{\sigma_{#1}^{#2}}
\newcommand\siginv[1]{\sigma_{#1}^{-1}}
\newcommand\Sii{\mathfrak{S}_\infty^{\mathsf{sm}}}
\newcommand\Sisp{\mathfrak{S}_\infty^{\mathsf{sp}}}
\renewcommand\ss{s}
\renewcommand\SS{S}
\newcommand\Sym[1]{\mathfrak{S}_{\HS{-0.4}#1}}
\newcommand\TERM[2]{\mathsf{Term}_{#1}(#2)}
\newcommand\TT{T}

\newcommand\uu{u}

\def\VR(#1,#2){\vrule width0pt height#1mm depth#2mm}

\newcommand\vv{v}

\newcommand\wdots{, ..., }
\newcommand\ww{w}

\newcommand\xx{x}
\newcommand\XX{X}

\newcommand\yy{y}

\newcommand\zz{z}

\newcommand\ZZZZ{\mathbb{Z}}
\newcommand\ZZZZp{\mathbb{Z}_{>0}}

\author{Patrick DEHORNOY}

\address{Laboratoire de Math\'ematiques Nicolas Oresme UMR 6139\\ Universit\'e de Caen, 14032~Caen, France\\patrick.dehornoy@unicaen.fr\\{\textrm URL:} dehornoy.users.lmno.cnrs.fr}

\title{THE BRAID SHELF}

\keywords{selfdistributivity, shelf, braid, special braid, symmetric group, Laver table, charged braid, parenthesied braid}

\subjclass{20F36, 20N02, 57M25, 57M60}

\begin{document}

\begin{abstract}
The braids of~$\Bi$ can be equipped with a selfdistributive operation~$\op$ enjoying a number of deep properties. This text is a survey of known properties and open questions involving this structure, its quotients, and its extensions.
\end{abstract}

\maketitle




\section{Introduction}

According to an approach that can be traced back to Joyce~\cite{Joy}, Matveev~\cite{Mat}, and Brieskorn~\cite{Bri}, selfdistributivity (SD) is an algebraic distillation of Reidemeister's move of type~$\mathrm{III}$ and, therefore, it is not surprising that structures involving SD operations often appear in low-dimensional topology, typically when one wishes to attach isotopy-invariant colourings to the strands of a diagram. In this approach, which, for instance, leads to the fundamental quandle of a knot and to a number of nontrivial invariants~\cite{CKS, CJKLS, CarterSurvey}, SD structures are external tools outside the world of topological objects. However, there also exists a disjoint approach, in which the topological objects themselves are equipped with an SD operation: this happens for the braids of~$\Bi$, and for the elements of various related structures. As can be expected, combining the internal and external aspects of SD is what leads to interesting results. Our aim is to explore this a priori strange situation. 

Most of the results presented below already appeared in literature (with the exception of some in Section~\ref{S:Quot}), in particular in~\cite{Dhr}, where the connection between the SD operation on braids and the standard braid ordering is explained. However, a number of side results are disseminated in various papers and not easily accessible. Moreover, SD operations mainly occur as auxiliary tools, and there existed no comprehensive presentation specifically concentrating on the SD operations themselves. These notes aim at filling this gap. A particular orientation toward open questions has been given, appealing for further research.

The text is organized in four sections after this introduction. In Section~\ref{S:BraidOp}, we recall the existence of a selfdistributive operation~$\op$ on the family~$\Bi$ of braids involving an unlimited number of strands, and its basic properties. In Section~\ref{S:Spec}, we concentrate on special braids, which are those braids obtained from the unit braid using the SD operation~$\op$. Special braids form a free left-shelf, and give rise to canonical decompositions for arbitrary braids. Section~\ref{S:Quot} is devoted to a few observations about SD operations in quotients of braid groups, typically in permutation groups. Finally, we discuss in Section~\ref{S:Ext} three examples of SD operations living in extensions of the group~$\Bi$, and leading to further open questions.


\section{A selfdistributive operation on braids}\label{S:BraidOp}

The central object of interest in these notes is a certain binary operation~$\op$ defined on the braid group~$\Bi$. After recalling the standard terminology for selfdistributive structures, we introduce the operation~$\op$ in Subsection~\ref{SS:BraidOp}. In Subsection~\ref{SS:Thompson}, we briefly recall its origin as a projection from a certain geometry group of selfdistributivity. Next, we state a few typical algebraic properties of~$\op$ (Subsection~\ref{SS:Algebraic}), and we explain how it can be used to color the strands of braid diagrams (Subsection~\ref{SS:Action}).

\subsection{A self-distributive operation on~$\Bi$}
\label{SS:BraidOp}

The selfdistributivity law comes in two versions:\begin{gather}\label{LD}\tag{LD}
\text{\emph{left} selfdistributivity:}\qquad \xx \op (\yy \op \zz) = (\xx \op \yy) \op (\xx \op \zz),\\
\label{RD}\tag{RD}
\text{\emph{right} selfdistributivity:}\qquad (\xx \opR \yy) \opR \zz = (\xx \opR \zz) \opR (\yy \opR \zz).
\end{gather}
It is often more convenient in topology to use the right version (and, accordingly, the symbol~$\opR$), but, here, in order both to be coherent with the existing literature and because of some specific features that are not invariant under symmetry (see Remark~\ref{R:NotSym}), we shall use the left version, and the symbol~$\op$. We use the prefix ``left'' everywhere to avoid ambiguity.

\begin{defi}[\bf selfdistributive structures]
\ITEM1 A \emph{left-shelf} is a structure~$(\SS, \op)$, where $\op$ is a binary operation on~$\SS$ that obeys the law~$\LD$. 

\ITEM2 A \emph{left-spindle} is a left-shelf~$(\SS, \op)$,  in which $\xx \op \xx = \xx$ always holds.

\ITEM3 A \emph{left-rack} is a left-shelf~$(\SS, \op)$, in which left translations are bijective, that is, for every~$\aa$ in~$\SS$, the map $\LL_\aa: \yy \mapsto \aa \op \yy$ is a bijection from~$\SS$ to~itself.

\ITEM4 A \emph{left-quandle} is a left-rack that is also a left-spindle.
\end{defi}

Lots of examples are known. We refer to~\cite{Diz} for a general picture of the structures known so far, and, in particular, for a survey of results about (left) shelves that are not racks or spindles.

In a non-associative context, paying attention to brackets is necessary, and, for~$\aa$ in a set equipped with a binary operation~$\op$, we write $\aa^{[\mm]}$ and $\aa_{[\mm]}$ for the $\mm$th {\it right} and {\it left} powers of~$\aa$ inductively defined by
\begin{equation}\label{E:Powers}
\aa^{[1]} := \aa_{[1]} := \aa, \quad \aa^{[m+1]} := \aa \op \aa^{[m]}, \quad \aa_{[m+1]} := \aa_{[m]} \op \aa.
\end{equation}

Let us now turn to the braid operation. The braid group~$\BR\nn$ is the group of isotopy classes of $\nn$-strand braid diagrams, as well as the mapping class group of an $\nn$-punctured disk---see for instance~\cite{Bir} or~\cite{Dhr}. It admits the presentation
\begin{equation}\label{E:BraidPres}
\bigg\langle \sig1 \wdots \sig{\nn - 1} \ \bigg\vert\ 
\begin{matrix}
\sig\ii \sig j = \sig j \sig\ii 
&\text{for} &\vert i-j \vert\ge 2\\
\sig\ii \sig j \sig\ii = \sig j \sig\ii \sig j 
&\text{for} &\vert i-j \vert = 1
\end{matrix}
\ \bigg\rangle.
\end{equation}
For every~$\nn$, the inclusion of~$\{\sig1 \wdots \sig{\nn - 1}\}$ into~$\{\sig1 \wdots \sig\nn\}$ extends into an embedding~$\ii_{\nn, \nn + 1}$ of~$\BR\nn$ into~$\BR{\nn + 1}$. The group~$\Bi$ is the limit of the inductive system so obtained, hence simply the union of all~$\BR\nn$s when $\ii_{\nn, \nn + 1}$ is identified with identity. The group~$\Bi$ admits the presentation analogous to~\eqref{E:BraidPres} with an infinite sequence of generators~$\sig1, \sig2, ...$ 

From the presentation, it is clear that mapping~$\sig\ii$ to~$\sig{\ii + 1}$ for every~$\ii$ defines an endomorphism~$\sh$ of~$\Bi$, hereafter called the \emph{shift} endomorphism. Note that $\sh$ is not surjective: $\sig1$ does not lie in~$\Im(\sh)$.

\begin{defi}[\bf shifted conjugation]\label{D:BraidOp}
For~$\br_1, \br_2$ in~$\Bi$, we put
\begin{equation}\label{E:BraidOp}
\br_1 \op \br_2:= \br_1 \cdot \sh(\br_2) \cdot \sig1 \cdot \sh(\br_1)\inv.
\end{equation}
\end{defi}

The braid $\br_1 \op \br_2$ is a sort of conjugate of~$\br_2$ under~$\br_1$, but with shifts and an additional factor~$\sig1$ added, see Fig.~\ref{F:BraidOp}. The reader can check the values  
$$1 \op 1 = \sig1, \quad 1 \op \sig1 = \sig2\sig1, \quad \sig1 \op 1 = \sigg12\siginv2, \quad ...$$
and, more generally, $1^{[\mm]} = \sig{\mm - 1} {\pdots} \sig2\sig1$ for every~$\mm \ge 1$. The above equalities show that $\op$ is neither commutative nor idempotent. Note that, because of the shift operator in~\eqref{E:BraidOp}, the operation~$\op$ is defined on~$\Bi$ only, and it induces no well-defined operation on~$\BR\nn$ for any finite~$\nn$.

\begin{figure}[htb]
$$\begin{picture}(74,19)(0,0)
\psline[linewidth=2pt]{c->}(0,0)(43,0)(49,6)(74,6)
\psline[linewidth=2pt,border=4pt]{c->}(0,6)(43,6)(49,0)(74,0)
\psline[linewidth=2pt]{c->}(0,12)(74,12)
\psline[linewidth=2pt]{c->}(0,18)(74,18)
\psframe[fillstyle=solid,fillcolor=white,linecolor=white](5,-2)(20,20)\psline[linewidth=0.5pt](5,20)(5,-2)(20,-2)(20,20)
\psframe[fillstyle=solid,fillcolor=white,linecolor=white](25,4)(40,20)\psline[linewidth=0.5pt](25,20)(25,4)(40,4)(40,20)
\psframe[fillstyle=solid,fillcolor=white,linecolor=white](52,4)(67,20)\psline[linewidth=0.5pt](52,20)(52,4)(67,4)(67,20)
\put(11,8){$\br_1$}
\put(12,19){$\vdots$}
\put(31,11.5){$\br_2$}
\put(32,19){$\vdots$}
\put(57,11.5){$\br_1\inv$}
\put(59,19){$\vdots$}
\end{picture}$$
\caption{\sf Braid diagram for $\br_1 \op \br_2$: a sort of conjugation of~$\br_2$ under~$\br_1$, with shifts and one~$\sig1$ added.}
\label{F:BraidOp}
\end{figure}

Our interest in the operation~$\op$ on~$\Bi$ stems from

\begin{prop}[\bf braid shelf]
The operation of~\eqref{E:BraidOp} obeys the law~$\LD$, that is, $(\Bi, \op)$ is a left-shelf. It is neither a left-spindle, nor a left-rack.
\end{prop}

\begin{proof}
A simple verification. Expanding the definition, we find for all $\br_1, \br_2, \br_3$
\begin{align*}
\br_1\op (\br_2 \op \br_3) 
&= \br_1\cdot \sh(\br_2) \cdot \sh^2(\br_3) \cdot \sig2 \cdot \sh^2(\br_2)\inv \cdot \sig1 \cdot \sh(\br_1)\inv,\\
(\br_1\op \br_2) \op (\br_1\op \br_3)
&= (\br_1\cdot \sh(\br_2) \cdot \sig1 \cdot \sh(\br_1)\inv) \op (\br_1\cdot \sh(\br_3) \cdot \sig1 \cdot \sh(\br_1)\inv)\\
&= (\br_1\cdot \sh(\br_2) \cdot \sig1 \cdot \sh(\br_1)\inv) \cdot \sh(\br_1\cdot \sh(\br_3) \cdot \sig1 \cdot \sh(\br_1)\inv)\\
&\hspace{3cm}\cdot \sig1 \cdot \sh(\br_1\cdot \sh(\br_2) \cdot \sig1 \cdot \sh(\br_1)\inv)\inv\\
&= \br_1\cdot \sh(\br_2) \cdot \sig1 \cdot \sh(\br_1)\inv \cdot \sh(\br_1)\cdot \sh^2(\br_3) \cdot \sig2 \cdot \sh^2(\br_1)\inv\\
&\hspace{3cm}\cdot \sig1 \cdot \sh^2(\br_1)\cdot \siginv2 \cdot \sh^2(\br_2)\inv \cdot \sh(\br_1)\inv\\
&= \br_1\cdot \sh(\br_2) \cdot \sig1 \cdot \sh^2(\br_3) \cdot \sig2 \cdot \sh^2(\br_1)\inv\\
&\hspace{3cm}\cdot \sig1 \cdot \sh^2(\br_1)\cdot \siginv2 \cdot \sh^2(\br_2)\inv \cdot \sh(\br_1)\inv.
\end{align*}
As $\sig1$ commutes with every braid in~$\Im(\sh^2)$ and $\sig1 \sig2 \sig1 \siginv2 = \sig2 \sig1$ holds, we find
$$\br_1\op (\br_2 \op \br_3) = (\br_1\op \br_2) \op (\br_1\op \br_3) = \br_1\cdot \sh(\br_2) \cdot \sh^2(\br_3) \cdot \sig2\sig1 \cdot \sh^2(\br_2)\inv \cdot \sh(\br_1)\inv,$$
and the law $\LD$ is satisfied.

We already noted the equality $1 \op 1 = \sig1$, which shows that $(\Bi, \op)$ is not a left spindle. Observe that, more generally, we \emph{always} have
\begin{equation}\label{E:NotIdem}
\br \not= \br \op \br,
\end{equation}
as an equality would expand into $1 = \sh(\br) \sig1 \sh(\br)\inv$, clearly imposible.

On the other hand, we have $1 \op \br = \sh(\br) \, \sig1$. As $\sig1$ does not belong to the image of~$\sh$, it is impossible to have $1 \op \br = 1$, so the left translation associated with~$1$ is not bijective, and $(\Bi, \op)$ is not a left-rack.
\end{proof}

\begin{rema}
Of course, we can obtain a right selfdistributive operation~$\opR$ by considering the opposite operation, namely
\begin{equation}\label{E:BraidOpR}
\br_1 \opR \br_2:= \sh(\br_2)\inv \cdot \sig1 \cdot \sh(\br_1) \cdot \br_2.
\end{equation}
\end{rema}

\subsection{Where does the braid shelf come from?}\label{SS:Thompson}

The braid operation~$\op$ was introduced in Def.~\ref{D:BraidOp} without any explanation, and Formula~\eqref{E:BraidOp} may look odd. We refer to~\cite{Dgb} for a complete explanation, but a few words here could be welcome. 

One can naturally associate with every algebraic law, or family of algebraic laws, a certain ``geometry monoid'' that, in some sense, captures the specific properties of the involved laws. When the law is the associativity law $x(yz) = (xy)z$, the geometry monoid is essentially Richard Thompson's group~$F$~\cite{CFP, Dhb} and, similarly, when both associativity and commutativity are considered, the geometry monoid is Thompson's group~$V$.

In the case of the selfdistributivity law~$\LD$, the geometry monoid turns out to be closely connected with a certain group~$\GLD$ generated by an infinite family of elements~$\LDop\alpha$ indexed by finite sequences of~$0$s and~$1$s and subject to an explicit list of relations~$\Rel\LD$. A basic property of~$\LD$ says that, if a left-shelf~$(\SS, \op)$ is generated by a single element~$\gg$, then, for every~$\aa$ in~$\SS$, the equality 
\begin{equation}\label{E:Abs}
\gg^{[\nn + 1]} = \aa \op \gg^{[\nn]}
\end{equation}
holds for~$\nn$ large enough: the result is trivial when~$\aa$ is~$\gg$ and the inductive argument
\begin{equation}\label{E:Abs1}
\gg^{[\nn + 1]} = \aa \op \gg^{[\nn]} = \aa \op (\bb \op \gg^{[\nn - 1]}) = (\aa \op \bb) \op (\aa \op \gg^{[\nn - 1]}) = (\aa \op \bb) \op \gg^{[\nn]}
\end{equation}
gives~\eqref{E:Abs} for~$\aa \op \bb$ starting from~\eqref{E:Abs} for~$\aa$ and for~$\bb$.

By construction, every formal consequence of the law~$\LD$ is encoded in an element of the group~$\GLD$. From there, the four equalities in~\eqref{E:Abs1} are encoded in a product of four elements in~$\GLD$, and they correspond to defining on~$\GLD$ a binary operation~$\opG$ by
\begin{equation}\label{E:GeomOp}
\ff \opG \gg:= \ff \cdot \sh_1(\gg) \cdot \LDop\emptyset \cdot \sh_1(\ff)\inv,
\end{equation}
where $\sh_1$ is the endomorphism that maps each element~$\LDop\alpha$ to~$\LDop{1\alpha}$ (appending an initial~$1$ in the finite sequence~$\alpha$). The operation~$\opG$ on~$\GLD$ is not selfdistributive, but, due to the connection of~$\opG$ with~\eqref{E:Abs}, the ``$\LD$-defect'' of~$\opG$, namely the quotient
\begin{equation*}
(\ff \opG (\gg \opG \hh))\inv \cdot ((\ff \opG \gg) \opG (\ff \opG \hh))
\end{equation*}
must lie in the image of the endomorphism~$\sh_0$ that maps~$\LDop\alpha$ to~$\LDop{0\alpha}$ for each~$\alpha$. As a consequence, when the subgroup~$\sh_0(\GLD)$ is collapsed, the operation~$\opG$ on~$\GLD$ must induce on the quotient an operation with no $\LD$-defect, that is, a selfdistributive operation.

It should not be a surprise to hear that the quotient-group $\GLD{/}\sh_0(\GLD)$ is (isomorphic to) the braid group~$\Bi$, and that the operation induced by~$\opG$ on~$\Bi$ is the operation~$\op$ of~\eqref{E:BraidOp}. More precisely, collapsing $\sh_0(\GLD)$ amounts to collapsing all elements~$\LDop\alpha$ such that $\alpha$ contains~$0$. All relations of~$\Rel\LD$ then become trivial, except the following ones:
\begin{gather*}
\LDop{1^\ii} \, \LDop{1^\jj} \, \LDop{1^\ii} = \LDop{1^\jj} \, \LDop{1^\ii} \,\LDop{1^\jj} \, \LDop{1^\ii0} \qquad\text{for $\jj = \ii + 1 \ge 1$}\\ 
\LDop{1^\ii} \, \LDop{1^\jj} = \LDop{1^\jj} \, \LDop{1^\ii} \qquad\text{for $\jj \ge \ii + 2 \ge 2$.}
\end{gather*}
Then collapsing $\LDop{1^\ii0}$ and mapping~$\LDop{1^\ii}$ to~$\sig{\ii+1}$ defines the expected epimorphism from~$\GLD$ onto~$\Bi$. Then \eqref{E:GeomOp} projects to~\eqref{E:BraidOp}, since $\sh_1$ projects to~$\sh$, and $\LDop\emptyset$ is mapped to~$\sig1$. Thus, the braid operation~$\op$ does not come out of the blue.

\subsection{Algebraic properties}\label{SS:Algebraic}

The left-shelf~$(\Bi, \op)$ is rather different from usual selfdistributive structures such as the various racks and quandles appearing in topology. Here we mention some of its algebraic properties, mainly those involving left translations, which are typical both in their statement and in their proof. 

We already noted that $(\Bi, \op)$ is not a left-rack, since the left translations $\LL_\br: \yy \mapsto \br \op \yy$ need not be bijections. More is known about such translations~\cite{Dgb}.

\begin{prop}[\bf left translations]\label{P:LeftTr}
\ITEM1 For every braid~$\br$, the left translation~$\LL_\br$ is injective, that is, $(\Bi, \op)$ is left cancellative.

\ITEM2 A braid~$\brr$ lies in the image of~$\LL_\br$ if, and only if, $\br \op \brr = \br^{[2]} \op \brr$ holds.
\end{prop}

\begin{prproof}
\ITEM1 Expanding $\br \op \brr =  \br \op \brr' $ gives 
$$\br \cdot \sh(\brr) \cdot \sig1 \cdot \sh(\br)\inv = \br \cdot \sh(\brr')  \cdot \sig1 \cdot \sh(\br)\inv,$$
whence $\sh(\brr) = \sh(\brr') $ as cancellation is legal in the group~$\Bi$. This implies $\brr = \brr' $, as $\sh$ is injective. We deduce that $(\Bi, \op)$ is left cancellative.

\ITEM2 Assume $\brr = \br \op \br'$. The selfdistributivity law implies
$$\br \op \brr = \br \op (\br \op \br') = (\br \op \br) \op (\br \op \br') = \br^{[2]} \op \brr,$$ 
so the condition is necessary. The converse is more tricky. We begin with a general auxiliary result, namely the fact that, for every~$\br$ in~$\Bi$, 
\begin{equation}\label{E:Trick}
\text{$\br$ belongs to~$\sh(\Bi)$ \qquad  if, and only if, \qquad $\sh(\br)$ and $\sig1$ commute.}
\end{equation}
Indeed, if $\br$ belongs to the image of~$\sh$, then $\sh(\br)$ belongs to the image of~$\sh^2$, hence it commutes with~$\sig1$ in the group~$\Bi$. Conversely, for every~$\br$ in~$\BR\nn$, the handle trick of Fig.~\ref{F:Trick} gives
$$\sh(\br)\inv \,\siginv1\, \sh(\br) \,\sig1 = \sh(\br)\inv \, \sig2 \pdots \sig\nn\, \br
\,\siginv\nn \pdots \siginv2.$$
So, if $\sh(b)$ and $\sig1$ commute, we obtain $\br = \siginv\nn \dots \siginv 2\, \sh(\br) \, \sig 2 \pdots \sig\nn$, which belongs to~$\sh(\Bi)$ explicitly.

Now assume $\br \op \brr = \br^{[2]} \op \brr$. Expanding the expressions gives
$$\br \, \sh(\brr) \, \sig1 \, \sh(\br)\inv = \br \, \sh(\br) \sig1 \, \sh(\br)\inv \, \sh(\brr) \, \sig1 \, \sh^2(\br) \,  \siginv2 \, \sh^2(\br)\inv \, \sh(\br)\inv,$$
which can be rewritten as $\sh(\br\inv \brr \sh(\br)) \sig1 = \sig1 \sh(\br\inv \brr \sh(\br)) \sig1\siginv2$, whence, using the braid relations, into
$$\sh(\br\inv \, \brr \, \sh(\br) \, \siginv1) \cdot \sig1 = \sig1 \cdot \sh(\br\inv \, \brr \, \sh(\br) \, \siginv1),$$
which expresses that $\sh(\br\inv \, \brr \, \sh(\br)  \, \siginv1)$ and~$\sig1$ commute. By~\eqref{E:Trick}, this implies that $\br\inv \, \brr \, \sh(\br) \, \siginv1$ belongs to~$\sh(\Bi)$, hence there exists~$\br'$ satisfying \linebreak $\br\inv \, \brr \, \sh(\br) \, \siginv1 = \sh(\br')$. The latter equality is $\brr = \br \op \br'$. Hence, $\brr$ lies in the image of the left translation~$\LL_\br$.
\end{prproof}

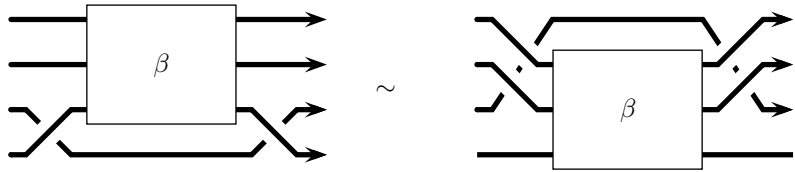
\begin{figure}[htb]
$$\begin{picture}(42,18)(0,0)
\psline[linewidth=2pt,border=4pt]{c->}(0,6)(2,6)(8,0)(32,0)(38,6)(42,6)
\psline[linewidth=2pt,border=4pt]{c->}(0,0)(2,0)(8,6)(32,6)(38,0)(42,0)
\psline[linewidth=2pt]{c->}(0,12)(42,12)
\psline[linewidth=2pt]{c->}(0,18)(42,18)
\psframe[fillstyle=solid,fillcolor=white,linewidth=0.5pt](10,4)(30,20)
\put(19,11){$\br$}
\put(48.5,8){$\sim$}
\end{picture}
\hspace{20mm}
\begin{picture}(42,18)(0,0)
\psline[linewidth=2pt](0,0)(42,0)
\psline[linewidth=2pt,border=4pt]{c->}(0,6)(2,6)(10,18)(30,18)(38,6)(42,6)
\psline[linewidth=2pt,border=4pt]{c->}(0,12)(2,12)(8,6)(32,6)(38,12)(42,12)
\psline[linewidth=2pt,border=4pt]{c->}(0,18)(2,18)(8,12)(32,12)(38,18)(42,18)
\psframe[fillstyle=solid,fillcolor=white,linewidth=0.5pt](10,-2)(30,14)
\put(19,5){$\br$}
\end{picture}$$
\caption{\sf The handle trick.}
\label{F:Trick}
\end{figure}

Again about left translations, let us mention another result, which is reminiscent of~\eqref{E:Abs}, but is quite different in that $(\Bi, \op)$ is not monogenerated.

\begin{prop}[\bf absorption]
A braid~$\br$ of~$\Bi$ belongs to~$\BR\nn$ if, and only if, the equality $\br \op 1^{[\nn]} = 1^{[\nn + 1]}$ is satisfied.
\end{prop}

\begin{proof}
By definition, we find in every case
$$\br \op 1^{[\nn]} = \br \cdot \sig\nn \pdots \sig2 \cdot \sig1 \cdot \sh(\br)\inv = \br \cdot 1^{[\nn + 1]} \cdot \sh(\br)\inv.$$
Assume $\br \in \BR\nn$. Then $\br$ can be expressed as a product of generators~$\sigg\ii{\pm 1}$ with $\ii < \nn$. Hence, by the handle trick of Fig.~\ref{F:Trick}, $\br\cdot \sig\nn \pdots \sig1$ is equal to~$\sig\nn \pdots \sig1\cdot \sh(\br)$, implying $\br \op 1^{[\nn]} = 1^{[\nn + 1]}$. 

Conversely, assume $\br \notin \BR\nn$. Let $\mm > \nn$ be minimal such that $\br$ belongs to~$\BR\mm$. By~\cite{Dfo}, we know that $\br$ admits an expression where exactly one of $\sig{\mm - 1}$, $\siginv{\mm - 1}$ occurs. It follows that $\sh(\br)$ has an expression where exactly one of $\sig\mm$, $\siginv\mm$ occurs, and the same holds for $\br \cdot 1^{[\nn + 1]} \cdot \sh(\br)\inv \cdot (1^{[\nn + 1]})\inv$, since $\br$, $1^{[\nn + 1]}$, and $(1^{[\nn + 1]})\inv$ can be expressed without~$\sigg{\mm}{\pm1}$. Hence, by~\cite{Dfb}, the latter braid cannot be the unit braid, that is, $\br \op 1^{[\nn]} = 1^{[\nn + 1]}$ fails.
\end{proof}

The case of right translations~$\RR_\br: \xx \mapsto \xx \op \br$ is different, since they need not be injective, but, again, the image can be characterized, this time using conjugacy instead of equality, see~\cite{Dgb}. 

\begin{rema}
Because of the similarity with conjugacy, the operation~$\op$ gives rise to potentially difficult algorithmic problems and, therefore, it might be used in cryptographic protocols~\cite{Dhk, KLT, LoU}. This remains marginal so far.
\end{rema}

\subsection{The partial action of braids on sequences of braids}\label{SS:Action}

If $(\SS, \op)$ is a left-rack and we use~$\opb$ for the inverse operation so that $\aa \op \bb = \cc$ is equivalent to $\aa \opb \cc = \bb$, it is well known that the formulas
\begin{gather}\label{E:Action1}
(\aa_1\wdots  \aa_\nn) \act \sig\ii = (\aa_1, \pdots , \aa_{\ii-1}, \aa_\ii \op \aa_{\ii+1}, \aa_\ii, \aa_{\ii+2}\wdots  \aa_\nn)\\
\label{E:Action2}
(\aa_1\wdots  \aa_\nn) \act \siginv\ii = (\aa_1\wdots  \aa_{\ii-1}, \aa_{\ii + 1}, \aa_{\ii+1}\opb \aa_\ii, \aa_{\ii+2}\wdots  \aa_\nn)
\end{gather}
define an action of~$\BR\nn$ on~$\SS^\nn$. This action can be described in terms of strand colorings: for $\br$ an $\nn$-strand braid and $\vec\aa$ a sequence in~$\SS^\nn$, the value of $\vec\aa \act \br$ is the sequence of output colors obtained when the input colors~$\vec\aa$ are attributed to the initial ends of the strands of~$\br$ and the colors are propagated according to the rules
$$
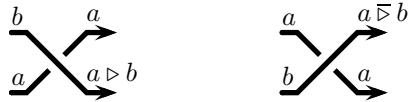
\begin{picture}(16,12)(0,0)
\psline[linewidth=2pt,border=3pt]{c->}(2,0)(4,0)(12,8)(16,8)
\psline[linewidth=2pt,border=3pt]{c->}(2,8)(4,8)(12,0)(16,0)
\put(2,1){$\aa$}
\put(2,9){$\bb$}
\put(12,9.5){$\aa$}
\put(12,1.5){$\aa\op\bb$}
\end{picture}
\hspace{20mm}
\begin{picture}(16,12)(0,0)
\psline[linewidth=2pt,border=3pt]{c->}(2,8)(4,8)(12,0)(16,0)
\psline[linewidth=2pt,border=3pt]{c->}(2,0)(4,0)(12,8)(16,8)
\put(2,1){$\bb$}
\put(2,9){$\aa$}
\put(12,9.5){$\aa\opb\bb$}
\put(12,1.5){$\aa$}
\end{picture}$$
When $(\SS, \op)$ is only a left-shelf, \eqref{E:Action2} need not make sense. Restricting to positive braids (those that can be expressed without any letter~$\siginv\ii$), and using~$\BP\nn$ for the monoid of $\nn$-strand positive braids, we obtain, see Fig.~\ref{F:Colouring}:

\begin{lemm}\label{L:ActionPos}
If $(\SS, \op)$ is a left-shelf, \eqref{E:Action1} defines an action of~$\BP\nn$ on~$\SS^\nn$.
\end{lemm}

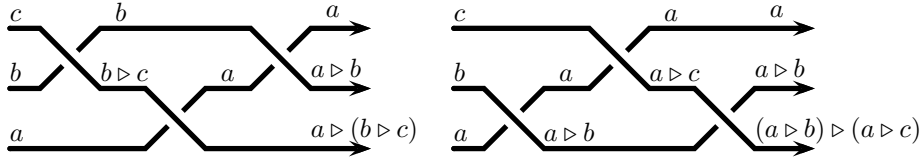
\begin{figure}[htb]
$$\begin{picture}(52,18)(0,0)
\psline[linewidth=2pt,border=3pt]{c->}(0,0)(18,0)(26,8)(32,8)(40,16)(48,16)
\psline[linewidth=2pt,border=3pt]{c->}(0,8)(4,8)(12,16)(32,16)(40,8)(48,8)
\psline[linewidth=2pt,border=3pt]{c->}(0,16)(4,16)(12,8)(18,8)(26,0)(48,0)
\put(0,1){$\aa$}
\put(0,9){$\bb$}
\put(0,17){$\cc$}
\put(12,9){$\bb\op\cc$}
\put(14,17){$\bb$}
\put(28,9){$\aa$}
\put(40,1.5){$\aa\op(\bb\op\cc)$}
\put(40,9.5){$\aa\op\bb$}
\put(42,17.5){$\aa$}
\end{picture}
\hspace{7mm}
\begin{picture}(58,18)(0,0)
\psline[linewidth=2pt,border=3pt]{c->}(0,0)(4,0)(12,8)(18,8)(26,16)(48,16)
\psline[linewidth=2pt,border=3pt]{c->}(0,8)(4,8)(12,0)(32,0)(40,8)(48,8)
\psline[linewidth=2pt,border=3pt]{c->}(0,16)(18,16)(26,8)(32,8)(40,0)(48,0)
\put(0,1){$\aa$}
\put(0,9){$\bb$}
\put(0,17){$\cc$}
\put(12,1){$\aa\op\bb$}
\put(14,9){$\aa$}
\put(28,17){$\aa$}
\put(26,9){$\aa \op\cc$}
\put(40,1.5){$(\aa\op\bb)\op(\aa\op\cc)$}
\put(40,9.5){$\aa\op\bb$}
\put(42,17.5){$\aa$}
\end{picture}$$
\caption{\sf Standard translation of Reidemeister~$\mathrm{III}$ move into the language of selfdistributivity: when colours from a set~$\SS$ are put on the left ends of the strands and then propagated so that a $\bb$-coloured strand becomes $\aa \op \bb$-coloured when it overcrosses an $\aa$-coloured arc, then the output colours are invariant under Reidemeister~$\mathrm{III}$ move if, and only if, $\op$ obeys the law~$\LD$.}
\label{F:Colouring}
\end{figure}

However, whenever $(\SS, \op)$ is a left cancellative left-shelf, we can extend the action of~$\BP\nn$ into a \emph{partial} action of the braid group~$\BR\nn$ on~$\SS^\nn$: using~$\opb$ for the \emph{partial} operation on~$\SS$ such that $\aa \op \bb = \cc$ is equivalent to $\aa \opb \cc = \bb$ \emph{when such an element~$\bb$ exists}, \eqref{E:Action2} can still be used. The technical result that makes this partial action possibly useful is the following one, whose proof is nontrivial (see~\cite[Sec.~3.1]{Dgb}):

\begin{lemm}[\bf partial action]\label{L:Action}
If $(\SS, \op)$ is a left cancellative left-shelf, then \eqref{E:Action1}--\eqref{E:Action2} induce a partial action of braid words such that:

\ITEM1 For every finite family of braid words~$\ww \wdots \ww_\pp$, there exists a sequence~$\vec\aa$ in~$\SS^\nn$ such that $\vec\aa \act \ww_\ii$ is defined for every~$\ii$;

\ITEM2 If $\ww$ and~$\ww'$ represent the same braid of~$\Bi$ and both $\vec\aa \act \ww$ and $\vec\aa \act \ww'$ are defined, then they are equal.
\end{lemm} 

In the above context, $\vec\aa \act \br$ can be unambiguously defined to be~$\vec\aa \act \ww$ for any word~$\ww$ such that $\ww$ represents~$\br$ and $\vec\aa \act \ww$ is defined, if at least one exists.

As $(\Bi, \op)$ is a left cancellative left-shelf, it is eligible for the above (total) action of~$\BP\nn$, and partial action of~$\BR\nn$ on~$(\Bi)^\nn$. Such an action is crucial for the developments of Section~\ref{S:Spec}. For the moment, we observe that the external action of~$\BR\nn$ on~$\Bi^\nn$ can be connected with an internal multiplication, at the expense of introducing a shifted version of the product.

\begin{lemm}[\bf shifted product]\label{L:ActionMult}
For~$(\br_1\wdots  \br_\nn)$ in~$\Bi^n$, define
\begin{equation}
\shprod(\br_1\wdots  \br_\nn) := \br_1\cdot \sh(\br_2) \pdots \sh^{n-1}(\br_\nn).
\end{equation}
Then, for all $(\br_1\wdots  \br_\nn)$ in~$\Bi^\nn$ and~$\br$ in~$\BP\nn$, we have
\begin{equation}
\label{E:ShiftProd}
\shprod((\br_1 \wdots  \br_\nn) \act \br) = \shprod (\br_1 \wdots \br_\nn)  \cdot \br. 
\end{equation}
\end{lemm}

\begin{proof} 
For an induction, it suffices to consider the case $\br = \sig\ii$.  Next, because the action of~$\sig\ii$ keeps the first $\ii - 1$ entries fixed and is a shifting of the action of~$\sig1$, it is even sufficient to consider the case $\br = \sig1$. For $\nn = 2$, we find
\begin{align*}
\shprod((&\br_1, \br_2) \act  \sig1) 
= \shprod(\br_1 \op \br_2, \br_1)
= (\br_1 \op \br_2) \cdot \sh(\br_1) \\
&= \br_1 \cdot \sh(\br_2) \cdot \sig1 \cdot \sh(\br_1)\inv \cdot \sh(\br_1) = \br_1 \cdot \sh(\br_2) \cdot \sig1 = \shprod (\br_1, \br_2) \cdot \sig1,
\end{align*}
Finally, going from $\nn = 2$ to $\nn \ge 3$ amounts to append some entries $\sh^{\jj - 1}(\br_\jj)$ with $\jj \ge 3$ on the right of the shifted products. These entries are invariant under the action of~$\sig1$, and $\sig1$ commutes with all of them, so \eqref{E:ShiftProd} remains valid.
\end{proof} 

\begin{rema}\label{R:NotSym}
Formula~\eqref{E:ShiftProd} is instrumental in subsequent applications of~$\op$, in particular in the construction of a braid orderings in Section~\ref{SS:Order}. This is where using left selfdistributivity~$\LD$, and not its right counterpart~$\RD$, matters. With the right version, a shifted product starting from the right should be considered: the problem is that one needs to consider braid sequences of arbitrary length, with no fixed position to start from. It is certainly possible to translate the statements, but at the expense of losing naturalness.
\end{rema}

\section{Special braids}\label{S:Spec}

The braid shelf~$(\Bi, \op)$ is a large structure, about which little is known, see for instance Questions~\ref{Q:FreeFam} and~\ref{Q:Gen}. Here we consider a substructure of the braid shelf, namely the one generated (as a left-shelf) by the unit braid~$1$. This structure is much better understood, as we shall explain now.

We begin by recalling (without proof) a few general results about monogenerated left-shelves, in particular a useful freeness criterion (Subsection~\ref{SS:Mono}). Then special braids and the canonical braid decompositions they lead to are described in Subsection~\ref{SS:Spec}. Finally, we recall in Subsection~\ref{SS:Order} the connection between special braids and the canonical (Dehornoy) braid order, leading to the Laver conjecture, a deep open question.

\subsection{Monogenerated left-shelves}\label{SS:Mono}

Every braid~$\br$ generates under~$\op$ a substructure of~$(\Bi, \op)$, hence a sub-left-shelf. By definition, such left-shelves are monogenerated, that is, generated by a single element. This implies a number of consequences because of the so-called comparison property involving left division.

\begin{defi}[\bf division relation]\label{D:Div}
For $\op$ a binary operation on~$\SS$ and~$\aa, \bb$ in~$\SS$, we say that $\aa$ \emph{(left) divides}~$\bb$, written $\aa \pref \bb$, if $\aa \op \xx = \bb$ holds for some~$\xx$. We write~$\prefs$ for the transitive closure of~$\pref$.
\end{defi}

If $\op$ is associative, there is no need to distinguish between~$\pref$ and~$\prefs$, since we then have $(\aa \op \xx_1) \op \xx_2 = \aa \op (\xx_1 \op \xx_2)$, but, in general, $\pref$ need not be transitive.

The following result about selfdistributivity is fundamental:

\begin{lemm}[\bf comparison property]\label{L:Comp}
If $(\SS, \op)$ is a monogenerated left-shelf and~$\aa, \bb$ belong to~$\SS$, then at least one of $\aa \prefs \bb$, $\aa = \bb$, $\bb \prefs \aa$ holds.
\end{lemm}

If $\phi$ is a morphism, $\aa \prefs \bb$ implies $\phi(\aa) \prefs \phi(\bb)$, so the point is to prove Lemma~\ref{L:Comp} when $\SS$ is a free left-shelf with one generator. We use the result (which is nontrivial) as a black box, referring for instance to~\cite{Diz} for an idea of the proof. 

One of the interests of the comparison property is to provide a simple criterion for recognizing \emph{free} monogenerated left-shelves. We recall the formal definition:

\begin{defi}[\bf free family, free shelf]\label{D:Free}
If $(\SS, \op)$ is left-shelf, a subfamily~$\XX$ of~$\SS$ is said to be \emph{free} in~$\SS$ if, for every left-shelf~$\SS^\sharp$, every map from~$\XX$ to~$\SS^\sharp$ extends in a morphism from the subshelf of~$\SS$ generated by~$\XX$ to~$\SS^\sharp$. We say that $(\SS, \op)$ is \emph{free based on~$\XX$} if $\XX$ generates~$\SS$ and is free in~$\SS$.
\end{defi}

\begin{lemm}[\bf freeness criterion]\label{L:Freeness}
If $(\SS, \op)$ is a left-shelf generated by a single element~$\gg$ and division has no cycle in~$\SS$, then $(\SS, \op)$ is free based on~$\gg$. Moreover, the relation~$\prefs$ is a (strict) linear order on~$\SS$, and, for all~$\aa, \bb, \cc$ in~$\SS$, we have
\begin{equation}\label{E:Order}
\aa \prefs \aa \op \bb, \qquad \text{and}\qquad \bb \prefs \cc \ \Leftrightarrow\  \aa \op \bb \prefs \aa \op \cc.
\end{equation}
\end{lemm}

\begin{proof}
By definition, the relation~$\prefs$ is transitive, and the assumption that $\pref$ has no cycle implies that $\prefs$ is irreflexive ($\aa \prefs \aa$ is always false). Hence, $\prefs$ is a strict order on~$\SS$. The comparison property implies that this order is linear. The relation $\aa \pref \aa \op \bb$ holds by definition, hence so does~$\aa \prefs \aa \op \bb$. Moreover, $\bb \pref \cc$, say $\bb \op \xx = \cc$, implies $(\aa \op \bb) \op (\aa \op \xx) = \aa \op \cc$ by~$\LD$, hence $\aa \op \bb \pref \aa \op \cc$. Thus, $\bb \prefs \cc$ implies $\aa \op \bb \prefs \aa \op \cc$. The converse implication must holds, since $\prefs$ is a strict linear order: $\aa \op \bb \pref \aa \op \cc$ excludes $\bb = \cc$ and $\bb \mults \cc$, so $\bb \prefs \cc$ is the only possibility. So~\eqref{E:Order} is satisfied.

Let $\SS^\sharp$ be an arbitrary left-shelf, and let $\gg^\sharp$ lie in~$\SS^\sharp$. By assumption, every element of~$\SS$ is the evaluation at~$\gg$ of some expression~$\TT(\xx)$ involving a variable~$\xx$ and~$\op$ (a ``term''). We construct a morphism~$\phi$ from~$\SS$ to~$\SS^\sharp$ by mapping~$\TT(\gg)$, that is, $\TT(\xx)$ evaluated at~$\gg$ in~$\SS$, to~$\TT(\gg^\sharp)$, that is, $\TT(\xx)$ evaluated at~$\gg^\sharp$ in~$\SS^\sharp$. The problem is that an element of~$\SS$ is the evaluation of several terms, whose evaluations need not a priori coincide.

Let $\TERM\op\xx$ be the family of all terms constructed from~$\xx$ and~$\op$, and let~$\eqLD$ be the congruence on~$\TERM\op\xx$ generated by the instances of the law~$\LD$: two terms~$\TT, \TT'$ are $\eqLD$-equivalent if, and only if, one can go from~$\TT$ to~$\TT'$ by repeatedly applying the law~$\LD$. If $\TT \eqLD \TT'$ holds, the assumption that $(\SS^\sharp, \op)$ obeys the law~$\LD$ implies $\TT(\gg^\sharp) = \TT'(\gg^\sharp)$ in~$\SS^\sharp$. Now assume $\TT \not\eqLD \TT'$. By construction, $\TERM\op\xx{/}{\eqLD}$ is a monogenerated left-shelf, hence it satisfies the comparison property. The assumption $\TT \not\eqLD \TT'$ means that the classes of~$\TT$ and~$\TT'$ do not coincide, hence they must be connected by~$\prefs$ or~$\mults$. Assume the former. By definition, this means that there exists~$\nn \ge 1$ and terms~$\TT_1 \wdots \TT_\nn$ satisfying
$$( \pdots ((\TT \op \TT_1) \op \TT_2) \ \op \pdots ) \op \TT_\nn \eqLD \TT',$$
which in turn implies in the left-shelf~$\SS$
$$( \pdots ((\TT(\gg) \op \TT_1(\gg)) \op \TT_2(\gg)) \ \op \pdots ) \op \TT_\nn(\gg) \eqLD \TT'(\gg),$$
whence $\TT(\gg) \prefs \TT'(\gg)$ and, a fortiori, $\TT(\gg) \not= \TT'(\gg)$. So, finally, $\TT(\gg) = \TT(\gg')$ can occur only when $\TT \eqLD \TT'$ holds, and, therefore, it implies $\TT(\gg^\sharp) = \TT'(\gg^\sharp)$ in~$\SS^\sharp$. So the morphism~$\phi$ is well defined, and $\SS$ is free.  
\end{proof}

The previous criterion is important here because of the following results.

\begin{lemm}[\bf $\sig1$-positivity]\label{L:DivPos}
Call a braid \emph{$\sig1$-positive} if it admits at least one expression in which the letter~$\sig1$ occurs and no letter~$\siginv1$ does. Then, for all braids~$\br, \br'$, the relation $\br \prefs \br'$ in~$\Bi$ implies that $\br\inv \br'$ is $\sig1$-positive.
\end{lemm}

\begin{proof}
By definition, $\br \prefs \br'$ holds if, and only if, there exist~$\nn \ge 1$ and $\br_1 \wdots \br_\nn$ satisfying 
\begin{equation}\label{E:Acyclic1}
(\pdots ((\br \op \br_1) \op \br_2) \pdots ) \op \br_\nn = \br'.
\end{equation}
According to the definition of~$\op$, this expands into an equality of the form
\begin{equation}\label{E:Acyclic2}
\br\inv \br' = \sh(\brr_0) \, \sig1 \, \sh(\brr_1) \, \sig1\,  \pdots \, \sig1\, \sh(\brr_\nn),
\end{equation}
where the right hand term is explicitly $\sig1$-positive.
\end{proof}

\begin{lemm}[\bf no cycle]\cite{Dfb}\label{L:Acyclic}
Division has no cycle in~$(\Bi, \op)$.
\end{lemm}

\begin{prproof}
By Lemma~\ref{L:DivPos}, a cycle for~$\prefs$, hence a relation~$\br \prefs \br$, would provide a $\sig1$-positive braid that is trivial (equal to~$1$). Hence, for excluding~\eqref{E:Acyclic1}, it suffices to prove that a $\sig1$-positive braid is never trivial. 

Several arguments exist, see in particular~\cite{Dfb, Dfo}. The simplest argument is the one, due to D.\,Larue~\cite{Lra}, that appeals to the Artin representation of~$\Bi$ in~$\Aut(\Fi)$, where $\Fi$ is a free group based on an infinite family $\{\xx_\ii \mid \ii \ge 1\}$, identified with the family of freely reduced words on~$\{\xx_\ii^{\pm1} \mid \ii \ge 1\}$. Artin's representation is defined by the rules
\begin{equation*}
\rho(\sig\ii)(\xx_\ii) := \xx_\ii \xx_{\ii + 1} \xx_\ii\inv, \quad \rho(\sig\ii)(\xx_{\ii + 1}) := \xx_\ii, \quad \rho(\sig\ii)(\xx_\kk) := \xx_\kk \text{ for $\kk \not= \ii, \ii + 1$},
\end{equation*}
and simple arguments about free reduction show that, if $\brr$ is a $\sig1$-positive braid, then $\rho(\brr)$ maps~$\xx_1$ to a reduced word that finishes with the letter~$\xx_1\inv$ and, therefore, $\brr$ cannot be trivial, since $\rho(1)$ maps~$\xx_1$ to~$\xx_1$, which does not finish with~$\xx_1\inv$. 
\end{prproof}

Merging Lemma~\ref{L:Acyclic} with the criterion of Lemma~\ref{L:Freeness}, we deduce

\begin{prop}[\bf free]\label{P:Free}
For every braid~$\br$, the substructure~$\langle\br\rangle$ of~$(\Bi, \op)$ generated by~$\br$ is free based on~$\{\br\}$. Moreover, the division relation provides a linear order on~$\langle\br\rangle$ that satisfies~\eqref{E:Order}.
\end{prop}

Another consequence of Lemma~\ref{L:Acyclic} is that the sufficient freeness condition of Lemma~\ref{L:Freeness} is also necessary:

\begin{coro}\label{C:FreeAcyclic}
Division in a free left-shelf has no cycle.
\end{coro}

\begin{proof}
Let $(\SS, \op)$ be a free left-shelf based on~$\XX$. By the universal property of free left-shelves, the map from~$\XX$ to~$\Bi$ sending every element to the braid~$1$ extends to a morphism~$\pi$ from~$(\SS, \op)$ to~$(\Bi, \op)$. Then the image under~$\pi$ of a cycle for~$\pref$ in~$\SS$ would be a cycle for~$\pref$ in~$\Bi$. By Lemma~\ref{L:Acyclic}, such a cycle cannot exist. 
\end{proof}

\subsection{Special braids and special decompositions}\label{SS:Spec}

Hereafter, we concentrate on the particular case of the substructure of~$(\Bi, \op)$ generated by~$1$ (unit braid).

\begin{defi}[\bf special braid]\label{D:Special}
We denote by~$\Bisp$ the closure of~$\{1\}$ in~$(\Bi, \op)$. The elements of~$\Bisp$ are called \emph{special} braids.
\end{defi}

A braid is special if it admits an expression that exclusively involves the braid~$1$ and the operation~$\op$---in other words, if it is the evaluation at~$1$ of a term of~$\TERM\op\xx$ in~$(\Bi, \op)$. For instance, $1$, $\sig1$, $\sig2 \sig1$, $\sigg12 \sig2\inv$ are special braids, as we have 
$$\sig1 = 1 \op 1, \quad \sig2 \sig1 = 1 \op (1 \op 1), \quad \sigg12\sig2\inv = (1 \op 1) \op 1.$$
Similarly, the braid $\sig\mm \,{\pdots}\, \sig2\sig1$ is special, as we saw above that it is the right power~$1^{[\mm + 1]}$. On the other hand, lots of braids are not special: for instance, Lemma~\ref{L:PosSpec} below implies that $\sig\ii$ is not special for $\ii \ge 2$.

Let us begin with a closure property of~$\Bisp$.

\begin{lemm}[\bf closure]\label{L:SpecDiv}
Special braids are closed under left division, in the sense that, if $\br$ and~$\brr$ are special and $\br \op \br' = \brr$ holds, then $\br'$ is necessarily special.
\end{lemm}

\begin{prproof}
By the easy direction of Prop.~\ref{P:LeftTr}, the existence of~$\br'$ in~$\Bi$ satisfying $\br \op \br' = \brr$ implies $\br \op \brr = \br^{[2]} \op \brr$ in~$\Bi$, hence in~$\Bisp$. Conversely, it is known that, since $(\Bisp, \op)$ is a free monogenerated left-shelf, $\br \op \brr = \br^{[2]} \op \brr$ implies the existence in~$\Bisp$ of~$\br''$ satisfying $\br \op \br'' = \brr$: this is a highly nontrivial result about free shelves based on the existence of an explicit normal form~\cite{Dgd}. Then left cancellativity forces~$\br' = \br''$, hence $\br' \in \Bisp$.
\end{prproof}

As it is a substructure of~$(\Bi, \op)$, the structure $(\Bisp, \op)$ is a left-shelf, hence eligible for  Lemma~\ref{L:ActionPos}, and, moreover, it is left cancellative, hence eligible for Lemma~\ref{L:Action}. Therefore, there exists a (total) action of~$\BP\nn$ on~$(\Bisp)^\nn$, and a partial action of~$\BR\nn$ on~$(\Bisp)^\nn$. These actions are fundamental in the sequel. 

The first result is a characterization of special braids in terms of the action.

\begin{lemm}[{\bf special} {\it vs.} {\bf action}]\label{L:RecSpec}
(See Fig.~\ref{F:Spec}.) A braid~$\br$ is special if, and only if, $(1, 1, 1, \pdots) \act \br$ exists and is equal to~$(\br, 1, 1, \pdots)$.
\end{lemm}

\begin{proof}
First, we inductively show that the condition is satisfied for every special braid~$\br$. It is obvious for~$\br = 1$. For $\br = \br_1 \op \br_2$ with $\br_1, \br_2$ special, we find
\begin{align*}
(1, 1, \pdots) \act  \br
&= ((((1, 1, \pdots) \act  \br_1) \act  \sh(\br_2)) \act  \sig1) \act  \sh(\br_1)\inv \\
&= (((\br_1, 1, 1, \pdots) \act  \sh(\br_2)) \act  \sig1) \act  \sh(\br_1)\inv \\
&= ((\br_1, \br_2, 1, 1, \pdots) \act  \sig1) \act  \sh(\br_1)\inv \\
&= (\br, \br_1, 1, 1, \pdots) \act  \sh(\br_1)\inv = (\br, 1, 1, 1, \pdots).
\end{align*}
For the last step, the induction hypothesis for~$\br_1$ implies that $(1, 1, 1, \pdots) \act  \br_1$ is defined and equal to~$(\br_1, 1, 1, \pdots)$. By reversing the diagram, we deduce that that $(\br_1, 1, 1, \pdots) \act  \br_1\inv$ is defined and equal to~$(1, 1, \pdots)$, and, from there, that $(\br, \br_1, 1, 1, \pdots) \act  \sh(\br_1)\inv$ is defined and equal to~$(\br, 1, 1, \pdots)$, as expected. 

Conversely, we claim that, whenever $(1, 1, 1, \pdots) \act \ww = (\br_1, \br_2, \pdots)$ holds, then all braids~$\br_\ii$ are special. The claim is true when $\ww$ is empty. It is preserved under adding a positive letter~$\sig\ii$, because special braids are closed under~$\op$, and it is preserved under adding a negative letter~$\siginv\ii$ because of Lemma~\ref{L:SpecDiv}. So, in particular, $(1, 1, 1, \pdots) \act \br = (\br, 1, 1, \pdots)$ implies that $\br$ is special.
\end{proof}

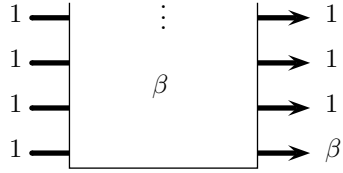
\begin{figure}[htb]
$$\begin{picture}(37,16)(0,0)
\psline[linewidth=2pt]{c->}(0,0)(37,0)
\psline[linewidth=2pt]{c->}(0,6)(37,6)
\psline[linewidth=2pt]{c->}(0,12)(37,12)
\psline[linewidth=2pt]{c->}(0,18)(37,18)
\psframe[fillstyle=solid,fillcolor=white,linecolor=white](5,-2)(30,20)\psline[linewidth=0.5pt](5,20)(5,-2)(30,-2)(30,20)
\put(16,8){$\br$}
\put(17,16.5){$\vdots$}
\put(-3,-0.5){$1$}
\put(-3,5.5){$1$}
\put(-3,11.5){$1$}
\put(-3,17.5){$1$}
\put(39,-0.5){$\br$}
\put(39,5.5){$1$}
\put(39,11.5){$1$}
\put(39,17.5){$1$}
\end{picture}$$
\caption{\sf A special braid is a braid that produces itself using braid coloring and starting from unit braids.}
\label{F:Spec}
\end{figure}

We now arrive at the main result, namely decompositions of arbitrary braids in terms of special braids.

\begin{prop}[\bf special decomposition]\label{P:Decomp}
\ITEM1  For every braid~$\br$ in~$\BP\nn$, there is a unique sequence of special braids~$\br_1 \wdots \br_\nn$ satisfying
\begin{equation}\label{E:Decomp1}
\br = \br_1 \cdot \sh(\br_2) \,\cdot \pdots \cdot\, \sh^{n-1}(\br_\nn).
\end{equation}
\ITEM2 For every braid~$\br$ in~$\BR\nn$, there are special braids~$\br_1 \wdots \br_\nn$, $\br'_1 \wdots \br'_\nn$ satisfying
\begin{equation}\label{E:Decomp2}
\br = \sh^{n-1}(\br_n\inv) \cdot \pdots \cdot \sh(\br_2\inv) \cdot {\br_1}\inv \cdot \br'_1 \cdot \sh(\br'_2) \cdot \pdots \cdot \sh^{n-1}(\br'_n).
\end{equation}
\end{prop}

\begin{proof} 
\ITEM1 Starting from~$\br$ in~$\BP\nn$, define $(\br_1 \wdots \br_\nn) := (1 \wdots 1) \act \br.$ As the input sequence $(1 \wdots 1)$ consists of special braids and $\Bisp$ is closed under~$\op$, all braids involved in the coloring of a positive diagram representing~$\br$ are special. So, in particular, the output colours $\br_1 \wdots \br_\nn$ are special. Next, we have $\shprod (1 \wdots 1) = 1$, so applying~\eqref{E:ShiftProd} directly gives~\eqref{E:Decomp1}. 

As for uniqueness, assume  $\br = \br'_1 \cdot \sh(\br'_2) \,\cdot \pdots \cdot\, \sh^{n-1}(\br'_\nn)$ with $\br'_1 \wdots \br'_\nn$ special. Then Lemma~\ref{L:RecSpec} implies $(1, 1, \pdots) \act \br'_\ii = (\br'_\ii, 1, 1, \pdots)$ for every~$\ii$. Using $\concat$ for concatenation of finite sequences, we deduce
\begin{align*}
(1, 1, \pdots) \act \br'_1 \cdot \sh(\br'_2) &= (\br'_1, 1, \pdots) \act \sh(\br'_2) = (\br'_1) \concat (1, 1, \pdots) \act \br'_2\\ 
&= (\br'_1) \concat (\br'_2, 1, 1, \pdots) = (\br'_1, \br'_2, 1, 1, \pdots),
\end{align*}
whence, repeating the argument,
$$(1, 1, 1, \pdots) \act \br'_1 \cdot \sh(\br'_2) \,\cdot \pdots \cdot\, \sh^{n-1}(\br'_\nn) = (\br'_1, \br'_2 \wdots \br'_\nn),$$
which is $(1, 1, 1, \pdots) \act \br = (\br'_1 \wdots \br'_\nn)$. It follows that $(\br'_1 \wdots \br'_\nn)$ must be the result of the action of~$\br$ on~$(1, 1, 1, \pdots)$, which is unique by Lemma~\ref{L:Action}.

\ITEM2 It is known that every braid in~$B_n$ can be expressed as a quotient $\br'{}\inv \cdot \br''$ with $\br',  \br''$ in~$\BP\nn$. The result then follows from applying~\ITEM1  to~$\br'$ and~$\br''$.
\end{proof} 

\begin{exam}\label{X:Decomp}
Consider $\br = \sig1^{-2}\sig2\sig1$, which is $\br'{}\inv\br''$ with $\br' =\nobreak \sigg12$ and $\br''  =\nobreak \sig2\sig1$. By~\eqref{E:Action1}, we find $(1,1,1) \act  \br' = (\sigg12\siginv2, \sig1, 1)$, and $(1,1,1) \act  \br''  = (\sig2\sig1, 1,1)$. We deduce the expression
$$\br = \sh^2(1)\inv \cdot \sh(\sig1)\inv \cdot (\sigg12\siginv2)\inv \cdot (\sig2 \sig1) \cdot \sh(1)
\cdot \sh^2(1),$$
or, using the trivial braid~$1$ and the operations~$\op$, $\inv$ and~$\sh$
exclusively,
\begin{equation}
\label{E:DecompX}
\br = \sh^2(1)\inv \cdot \sh(1\op1)\inv \cdot
((1\op1)\op1)\inv \cdot (1\op(1\op1)) \cdot \sh(1)
\cdot \sh^2(1),
\end{equation}
that is, when trivial terms are removed, $\sh(1^{[2]})\inv
\cdot (1_{[3]})\inv \cdot 1^{[3]}$.
\end{exam}

\begin{rema}\label{R:NonSym}
For $\br$ in~$\BP\nn$, Prop.~Ê\ref{P:Decomp} gives $\br = \br_1 \, \sh(\br_2) \pdots \sh^{\nn -1}(\br_\nn)$ with $\br_1 \wdots \br_\nn$ special. However, the braids $\br_\ii$ need not lie in~$\BR\nn$: for instance, the decomposition of~$\sigg12$, a braid of~$\BR2$, is $(\sigg12\siginv2) \, \sh(\sig1)$, with $\sigg12\siginv2 \notin \BR2$.
\end{rema}

The previous results imply that for a braid to be special is a decidable property.

\begin{prop}[\bf decidability]
There is an algorithm that decides whether a given braid word represents a special braid, and, if so, returns an expression of this braid in terms of~$1$ and~$\op$.
\end{prop}

\begin{prproof}
Let $\ww$ be a braid word. We can decide whether $\ww$ represents a special braid as follows. First, we reverse~$\ww$ into an equivalent braid word~$\uu\vv\inv$ with $\uu, \vv$ positive using the reversing method of~\cite{Dia}. Next, we compute $(1, 1, 1, \pdots) \act \uu\vv\inv$. By~\cite{Dfb}, it is known that, if $\vec\aa \act \ww$ is defined, then so is $\vec\aa \act \uu\vv\inv$ when $\uu$ and~$\vv$ are obtained as above, namely so as to ensure that the elements of~$\BPi$ representing~$\uu$ and~$\vv$ have no common right divisor. By Lemma~\ref{L:RecSpec}, $\ww$ represents a special braid if, and only if,  the computation is successful and it leads to a sequence of the form $(\br, 1, 1, \pdots)$, that is, all components from the second are trivial. The latter point can be tested using any algorithm for the word problem of braids. Moreover, there exists an effective left division algorithm in free monogenerated shelves~\cite{Dfd}. Hence, we can effectively obtain an expression of the special braids involved in~$(1, 1, 1, \pdots) \act \uu \vv\inv$ in terms of~$1$ and~$\op$.
\end{prproof}

\begin{exam}
Let $\ww:= \siginv2 \siginv1 \sigg22 \sig1$. Reversing~$\ww$ yields the equivalent positive--negative word~$\sigg12 \siginv2$, so, here, $\uu$ is $\sigg12$, and $\vv$ is~$\sig2$. Then we compute the value of $(1, 1, 1) \act \sigg12$, namely $(\sigg12 \siginv2, \sig1, 1)$, that is, $(1_{[3]}, 1^{[2]}, 1)$. Then, in order to apply~$\vv\inv$, we have to left divide $1^{[2]}$ by~$1$. In the present case, the result is obvious: division is possible and the quotient is~$1$. So we obtain $(1, 1, 1) \act \sigg12\siginv2 = (1_{[3]}, 1, 1)$, and we conclude that $\ww$ represents the special braid~$1_{[3]}$.
\end{exam}

We conclude with a characterisation first of braids that are positive and special, and of braids whose special decomposition only contains positive braids. 

\begin{lemm}[\bf positive special]\label{L:PosSpec}
For every~$\mm$, there is a unique special positive braid of length~$\mm$, namely~$1^{[\mm + 1]}$.
\end{lemm}

\begin{proof}
We saw above the equality $1^{[\mm + 1]} = \sig\mm \pdots \sig2\sig1$, a positive braid of length~$\mm$. Conversely, assume that $\br$ is positive and special of length~$\mm$. We use induction on~$\mm$. For $\mm = 0$, the only possibility is $\br = 1 = 1^{[1]}$. Assume $\mm \ge 1$, and write $\br = \br' \sig\ii$. By Lemma~\ref{L:RecSpec}, the assumption that $\br$ is special implies the equality $(1, 1, 1, \nobreak \pdots) \act \br = (\br, 1, 1, \dots)$. Since $\br'$ is positive, $(1, 1, \pdots) \act \br'$ exists. Call it $(\br'_1, \br'_2, \pdots)$. If we had $\ii \ge 2$, the $\ii$th entry in~$(1, 1, \pdots) \act \br' \sig\ii$ would be~$\br'_\ii \op \br'_{\ii+1}$, which cannot be~$1$. Thus we must have $\ii = 1$, whence $\br = \br'_1 \op \br'_2$, $\br'_1 = \br'_3 = \pdots = 1$, that is, $(1, 1, \pdots) \act \br' = (1, \br'_2, 1, 1, \pdots)$. By~\eqref{E:Decomp1}, we deduce $\br' = \sh(\br'_2)$. Hence $\br'_2$ is positive and special. The induction hypothesis implies $\br'_2 = 1^{[\mm]}$, and we deduce $\br = 1 \op 1^{[\mm]} = 1^{[\mm + 1]}$.
\end{proof}

\begin{prop}[\bf positive special decomposition]
A braid~$\br$ admits a special decomposition consisting of positive braids if, and only if, $\br$ is a positive simple braid, that is, there exists an integer~$\nn$ such that $\br$ divides Garside's fundamental braid~$\Delta_\nn$ in the monoid~$\BP\infty$.
\end{prop}

\begin{proof}
Assume $\br = \shprod(\br_1 \wdots \br_\nn)$ with $\br_1 \wdots \br_\nn$ special and positive. Then, as seen in the proof of Prop.~\ref{P:Decomp}, $(1, 1, 1, \pdots) \act \br$ is defined and equal to~$(\br_1, \br_2, \pdots)$. Hence, by Lemma~\ref{L:PosSpec}, there exists for each~$\ii$ a number~$\mm_\ii$ satisfying $\br_\ii = 1^{[\mm _\ii + 1]}$. Then \eqref{E:Decomp1} expands into
\begin{equation}\label{E:DecompSimple}
\br = (\sig{\mm_1} \pdots \sig1) \cdot \sh(\sig{\mm_2} \pdots \sig1) \,\cdot \pdots
\end{equation}
This shows that $\br$ is a positive braid, in which any two strands cross at most once, thus a divisor of~$\Delta_n$ for $n$ large enough.

Conversely, assume that $\br$ is a positive simple braid. Then it is known that $\br$ admits a decomposition of the form~\eqref{E:DecompSimple}, where $\mm_\ii + 1$ is the initial position of the strand that finishes at position~$\ii$ in~$\br$. By uniqueness, this decomposition coincides with the one associated with~$(1, 1, 1, \pdots) \act \br$, meaning that the entries of~$(1, 1, 1, \pdots) \act \br$ are the positive braids~$1^{[\mm_\ii + 1]}$.
\end{proof}

In particular, the special decomposition of~$\Delta_n$ is
$$\Delta_n = (\sig{n-1} \pdots \sig1) \cdot \sh(\sig{n-2} \pdots \sig1) \ \cdot \pdots \cdot\ \sh^{n-2}(\sig1).$$ 

In contrast to Lemma~\ref{L:PosSpec}, which completely describes special braids lying in~$\BP\nn$, very little is known about special braids that lie in~$\BR\nn$. Inductively defining the \emph{complexity} of a special braid~$\br$ by $\compl1 := 0$ and
$$\compl\br := \min\{\sup(\compl{\br_1}, \compl{\br_2}) + 1 \mid \br = \br_1 \op \br_2\},$$
one easily checks that $\compl\br \le \nn$ implies~$\br \in \BR\nn$.
\begin{ques}[\bf special $\nn$-strand braids]
{\it Is the converse true, that is, does $\compl\br \le \nn$ hold for every special braid~$\br$ that lies in~$\BR\nn$?}
\end{ques}

A positive answer would in particular imply that there are at most~$2^\nn$ special braids in~$\BR\nn$.

\subsection{Ordering braids}\label{SS:Order}

Together with the acyclicity result of Lemma~\ref{L:Acyclic}, the comparison property of Lemma~\ref{L:Comp} implies that the iterated division relation~$\prefs$ is a linear ordering on the family of all special braids~$\Bisp$. As every braid admits a decomposition in terms of special braids, it is not surprising that the linear order on~$\Bisp$ extends to a linear order on arbitrary braids. We briefly recall the construction. The interest here is not to establish the orderability of~$\Bi$, now a standard result~\cite{Dhr}, but to explain the connection with the shelf operation~$\op$ and state the Laver conjecture.

We recall from Lemma~\ref{L:DivPos} that a braid~$\br$ is said to be \emph{$\sig1$-positive} if it admits at least one expression where~$\sig1$ occurs and~$\siginv1$ does not. We define \emph{$\sig1$-negative}  (no~$\sig1$ and at least one~$\siginv1$) and \emph{$\sig1$-free} (no~$\sig1$ and no~$\siginv1$) accordingly.

\begin{lemm}[\bf order on special]\label{L:SpecOrd}
If $\br, \br'$ are special braids, the relation $\br \prefs \br'$ in~$\Bisp$ is equivalent to $\br\inv \br'$ being $\sig1$-positive.
\end{lemm}

\begin{proof}
We saw in Lemma~\ref{L:DivPos} that, if $\br \prefs \br'$ holds in~$\Bi$, hence in particular if $\br \prefs \br'$ holds in~$\Bisp$, then $\br\inv \br'$ is $\sig1$-positive. Conversely, assume that $\br$ and~$\br'$ are special and $\br\inv \br'$ is $\sig1$-positive. By Lemma~\ref{L:Comp}, at least one of $\br \prefs \br'$, $\br =\nobreak \br'$, or $\br \mults \br'$ holds in~$\Bisp$. The second option implies that $\br\inv \br'$ is trivial, hence, by Lemma~\ref{L:Acyclic}, not $\sig1$-positive. The third option implies that $\br'{}\inv \br$ is $\sig1$-positive, which again contradicts the $\sig1$-positivity of~$\br\inv \br'$, as the product of the $\sig1$-positive braids~$\br\inv \br'$ and~$\br'{}\inv \br$, which is~$1$, would be $\sig1$-positive. So $\br \prefs \br'$ is the only possibility.
\end{proof}

Using the special decomposition of Prop.~\ref{P:Decomp}, we immediately deduce:

\begin{prop}[\bf order]\label{P:Order}
Every braid is $\sig1$-positive, $\sig1$-negative, or $\sig1$-free, each possibility excluding the others.
\end{prop}

\begin{proof}
Let $\br \in \BR\nn$. By Prop.~\ref{P:Decomp}, there exist~$\br_1 \wdots \br_\nn, \br'_1 \wdots \br'_\nn$  special satisfying
\begin{equation*}
\br = \sh^{n-1}(\br_n\inv) \, \cdot \pdots \cdot\,  \sh(\br_2\inv) \cdot \br_1\inv \cdot \br'_1 \cdot \sh(\br'_2) \,\cdot \pdots \cdot\, \sh^{n-1}(\br'_n),
\end{equation*}
which has the form $\br = \sh(\brr) \cdot \br_1\inv\br'_1 \cdot \sh(\brr')$. By Lemma~\ref{L:SpecOrd}, $\br_1\inv\br'_1$ is either $\sig1$-positive, in which case so is~$\br$, or equal to~ $1$, in which case $\br$ is $\sig1$-free, or $\sig1$-negative, in which case so is~$\br$. The three cases exclude one another by Lemma~\ref{L:Acyclic}.
\end{proof}

\begin{rema}
The previous simple argument takes place in~$\Bi$, and does not guarantee that a braid of~$\BR\nn$ necessarily admits a $\sig1$-positive expression by an $\nn$-strand braid word. The latter property is true, but it requires a further proof~\cite{Dfo}.
\end{rema}

From that point, it is straightforward to define an order on~$\Bi$. 

\begin{coro}[\bf order]\label{C:Order}
Say that a braid is $\sig\ii$-positive if it is the image of a $\sig1$-positive braid under~$\sh^{\ii - 1}$. For~$\br_1, \br_2$ in~$\Bi$, define $\br_1 < \br_2$ to mean that $\br_1\inv \br_2$ is $\sig\ii$-positive for some~$\ii$. Then the relation~$<$ is a linear order on~$\Bi$; it is compatible with multiplication on the left, and extends the order~$\prefs$ on~$\Bisp$.
\end{coro}

It is then easy to check that the order so defined on~$\Bi$ is a lexicographic extension of the order~$\prefs$ on special braids: for every braid~$\br$, the relation $\br > 1$ holds if, and only if, whenever $\br_1 \wdots \br_\nn, \br'_1 \wdots \br'_\nn$ are special and we have
$$(\br_1 \wdots \br_\nn) \act \br = (\br'_1 \wdots \br'_\nn),$$
the sequence~$(\br_1 \wdots \br_\nn)$ is smaller than~$(\br'_1 \wdots \br'_\nn)$ with respect to the lexicographical extension of~$\prefs$ (look at the first~$\ii$ such that $\br_\ii$ and~$\br'_\ii$ do not coincide).

Let us return to the monoid~$\BPi$. The following was proved by R.\,Laver~\cite{Lvc}, and then made more precise by S.\,Burckel \cite{Bus} and J.\,Fromentin~\cite{Fro, FroJEMS, FrP}: 

\begin{prop}[\bf well-order]
For every~$\nn$, the restriction of the braid order~$<$ to~$\BP\nn$ is a well-order of order type the Cantor ordinal~$\omega^{\omega^{\nn - 2}}$; the restriction of~$<$ to~$\BPi$ is a well-order of order type~$\omega^{\omega^\omega}$. 
\end{prop}

If $\br$ is a positive $\nn$-strand braid, Lemma~\ref{L:ActionPos} guarantees that $(\br_1 \wdots \br_\nn) \act \br$ is defined for every sequence of braids~$(\br_1 \wdots \br_\nn)$. Thus, reversing the perspective and starting from the sequence~$(\br_1 \wdots \br_\nn)$, we see that the family of all braids~$\br$ for which $(\br_1 \wdots \br_\nn) \act \br$ is defined includes the monoid~$\BP\nn$.

\begin{ques}[\bf Laver conjecture]
{\it For every sequence of braids~$(\br_1 \wdots \br_\nn)$, is the family of all braids~$\br$ for which $(\br_1 \wdots \br_\nn) \act \br$ is defined well-ordered by~$<$?}
\end{ques}

R.\,Laver conjectured a positive answer. A similar question can be raised for every left-shelf~$(\SS, \op)$, and, in some cases, the family in question reduces to~$\BP\nn$~\cite{Lra}. But this is not the case in general, in particular in the case of~$\Bi$, and the question remains open, and it seems difficult. Note that this is a pure question of topology, in that it exclusively involves braids and no other structure.

\section{Quotients}\label{S:Quot}

We now summarize results about the selfdistributive operations obtained from the braid operation~$\op$ by quotienting the group or shelf structure. Although several natural and simple questions arise, little is known here. We successively consider the case of the symmetric group (Subsection~\ref{SS:Perm}) and of the Burau representation (Subsection~\ref{SS:Burau}), which arise when the group structure is quotiented, and the case of Laver tables, which arise when the shelf structure is quotiented (Subsection~\ref{SS:Laver}).

\subsection{Permutations}\label{SS:Perm}

For every~$\nn$, a well understood quotient of the group~$\BR\nn$ appears when the elements~$\sigg\ii2$ are collapsed, namely the symmetric group~$\Sym\nn$. Geometrically, the projection corresponds to associating with a braid~$\br$ the permutation~$\perm(\br)$ so that, for $1 \le \ii < \nn$, the strand finishing at position~$\ii$ begins at position~$\perm(\br)(\ii)$ in any braid diagram representing~$\br$. The projection extends to~$\Bi$, the image being the symmetric group~$\Si$ of all permutations of~$\ZZZZp$ that move finitely many entries, equipped with composition. 

We denote by~$\ss_\ii$ the projection of~$\sig\ii$, that is, the transposition exchanging~$\ii$ and~$\ii + 1$. We use~$\sh$ for the \emph{shift} endomorphism of~$\Si$ defined, for~$\ff$ in~$\Si$, by~$\sh(\ff)(1) := 1$ and $\sh(\ff)(\nn) := \ff(\nn - 1) + 1$ for~$\nn \ge 2$. 

As the operation~$\op$ on~$\Bi$ is defined from the group multiplication and the endomorphism~$\sh$, every group morphism from~$\Bi$ to a group~$\GG$ that carries~$\sh$ to some convenient endomorphism of~$\GG$ induces a $\op$-morphism, and the image operation automatically obeys the selfdistributivity law~$\LD$. In this way, we obtain:

\begin{prop}[\bf permutation shelf]
For~$\ff, \gg$ in~$\Si$, define
\begin{equation}
\ff \op \gg := \ff \cdot \sh(\gg) \cdot \ss_1 \cdot \sh(\ff)\inv.
\end{equation}
Then $(\Si, \op)$ is a left-shelf, and $\perm$ is a surjective morphism from~$(\Bi, \op)$ to~$(\Si, \op)$.
\end{prop}

On the shape of what was done for braids in Section~\ref{S:Spec}, it is natural to consider the substructure of~$(\Si, \op)$ generated by the identity map~$\id$ of~$\ZZZZp$.

\begin{defi}[\bf special permutation]
We denote by~$\Sisp$ the closure of~$\{\id\}$ in~$(\Si, \op)$. The elements of~$\Sisp$ are called \emph{special} permutations.
\end{defi}

Clearly, $(\Sisp, \op)$ is a monogenerated left-shelf, and $\perm$ induces a surjective morphism from~$(\Bisp, \op)$ onto~$(\Sisp, \op)$, since $\id$ is the permutation associated with the unit braid. One easily checks that, for every~Ê$\nn$, the right power~$\id^{[\nn]}$ is the $\nn$-cycle~$\ss_{\nn - 1} \pdots \ss_2\ss_1$. The left powers~$\id_{[\nn]}$ are more mysterious; the first values are
$$\id_{[2]} = \ss_1, \quad \id_{[3]} = \ss_2, \quad \id_{[4]} = \ss_2\ss_3\ss_1, \quad \id_{[5]} = \ss_3\ss_4\ss_2\ss_3\ss_4, \ ...$$
see Fig.~\ref{F:Perm} for an embryo of the table of~$(\Sisp, \op)$ starting with left powers.

\begin{figure}[htb]
$$\small
\begin{tabular}{c|ccccc}
$\Sisp$\VR(0,2)&\hbox to 15mm{\hfil$\id$\hfil}&\hbox to 15mm{\hfil$\ss_1$\hfil}&\hbox to 15mm{\hfil$\ss_2$\hfil}&\hbox to 15mm{\hfil$\ss_2\ss_3\ss_1$\hfil}&\hbox to 10mm{\hfil$\pdots$\hfil}\\
\hline
$\id$\VR(3,0)&$\ss_1$&$\ss_2\ss_1$&$\ss_3\ss_1$&$\ss_3\ss_4\ss_2\ss_1$&$\pdots$\\
$\ss_1$&$\ss_2$&$\ss_2\ss_1$&$\ss_3\ss_2$&$\ss_3\ss_4\ss_2\ss_1$&$\pdots$\\
$\ss_2$&$\ss_2\ss_3\ss_1$&$\ss_3\ss_1$&$\ss_2\ss_1$&$\ss_3\ss_4\ss_2\ss_3\ss_4\ss_1$&$\pdots$\\
$\ss_2\ss_3\ss_1$&$\ss_3\ss_4\ss_2\ss_3\ss_4$&$\ss_3\ss_4\ss_2\ss_3\ss_4\ss_1$&$\ss_4\ss_3$&$\ss_3\ss_1$&$\pdots$\\
$\pdots$&$\pdots$&$\pdots$&$\pdots$&$\pdots$&$\pdots$\\
\end{tabular}$$
\vspace{-5mm}
\caption{\sf An embryo of the table of~$(\Sisp, \op)$ based on left powers of~$\id$.}
\label{F:Perm}
\end{figure}

In the left-shelf~$(\Sisp, \op)$, the division relation~$\pref$ of Def.~Ê\ref{D:Div} admits cycles: for instance, one can check the equality 
\begin{equation}\label{E:Cycle}
\ss_2\ss_1 = ((\ss_2\ss_1) \ \op\  \ss_1) \ \op\  \ss_2\ss_3\ss_1,
\end{equation}
whence $\ss_2 \ss_1 \prefs \ss_2\ss_1$. It follows from Corollary~\ref{C:FreeAcyclic} that $(\Sisp, \op)$ is not a free left-shelf: typically, \eqref{E:Cycle} translates into $\id^{[3]} = (\id^{[3]} \op \id_{[2]}) \op \id_{[4]}$,  a nontrivial relation whose counterpart fails in~$\Bisp$.

\begin{ques}[\bf presentation]\label{Q:PermPres}
{\it Does the left-shelf~$(\Sisp, \op)$ admit a finite presentation in terms of its generator~$\id$?}
\end{ques}

Nothing is known. One of the very few results about~$(\Si, \op)$ known so far is an alternative definition of~$\op$ using conjugation of injections. 

\begin{prop}[\bf injection shelf]\label{P:Inj}
Let~$\Ii$ be the monoid of all injections from~$\ZZZZp$ to itself equipped with composition, and let~$\SH$ be the element of~$\Ii$ (``shift'') that maps~$\nn$ to~$\nn + 1$ for every~$\nn$. For~$\ff, \gg$ in~$\Ii$, define $\ff \op \gg$ in~$\Ii$ by
\begin{equation}\label{E:Inj}
\ff \op \gg(\nn):= \ff(\gg(\ff\inv(\nn))) \text{ for $\nn \in \Im(\ff)$, and } \ff \op \gg(\nn):= \nn \text{ otherwise}.
\end{equation}
Then $(\Ii, \op)$ is a left-shelf, and the map~$\phi: \ff \mapsto \ff \cdot \SH$ defines an embedding from~$(\Si, \op)$ into~$(\Ii, \op)$.
\end{prop}

\begin{proof}
For all~$\ff, \gg$, the definition implies $\Im(\ff \op \gg) = \ff(\Im(\gg)) \cup \coIm(\ff)$ and $\coIm(\ff \op \gg) = \ff(\coIm(\gg))$. From there, one easily checks that, for every~$\nn$, both $\ff \op (\gg \op \hh)$ and $(\ff \op \gg) \op (\ff \op \hh)$ map~$\nn$ to~$\ff(\gg(\hh(\gg\inv(\ff\inv(\nn)))))$ for~$\nn$ in~$\Im(\ff\cdot\gg)$, and to~$\nn$ otherwise. Thus $\op$ obeys the selfdistributivity law~$\LD$.

To prove that $\phi$ is a morphism, carefully applying the definitions shows that, for all~$\ff, \gg$ in~$\Si$, both $\phi(\ff \op \gg)$ and $\phi(\ff) \op \phi(\gg)$ keep~$\ff(1)$ fixed, and map~$\nn$ to~$\ff(\gg(\ff\inv(\nn)) + 1)$ for~Ê$\nn \not= \ff(1)$. Finally, $\phi(\ff) = \phi(\ff')$ implies $\sh(\ff) = \sh(\ff')$, whence $\ff = \ff'$, so $\phi$ is injective.
\end{proof}

By definition, the morphism~$\phi$ of Prop.~\ref{P:Inj} maps the identity permutation~$\id$ to the injection~$\SH$. Hence it induces an isomorphism from the left-shelf~$(\Sisp, \op)$ to its image, which is the substructure~$\Iisp$ of~$(\Ii, \op)$ generated by~$\SH$. The left-shelf~$(\Iisp, \op)$ is directly reminiscent of the left-shelf~$\Iter(\jj)$ constructed in set theory using the iterations of an elementary embedding~$\jj$ under the ``application'' operation, see~\cite{Lvb} or~\cite[Chapter~XII]{Dgb}: the latter structure implies the existence of an injection~$\SHt$ in~$\Ii$ and of a partial selfdistributive operation~$\opt$ on (a subset of)~$\Ii$ such that $\opt$ is everywhere defined on the substructure generated by~$\SHt$, and $\ff \opt \gg(\nn) = \ff(\gg(\ff\inv(\nn)))$ holds for $\nn$ in~$\Im(\ff)$, but, instead of ``stupidly'' completing with $\ff \op \gg(\nn) := \nn$ for~$\nn$ in~$\coIm(\ff)$ as in~\eqref{E:Inj}, the operation~$\opt$ is constructed so that $\ff \opt \gg$ is increasing when defined. The latter condition implies that the injection~$\SHt$ must be a fast growing function~\cite{Dou, DoJ, Jech}, and its existence is currently proved only from a large cardinal axiom (``Laver cardinal'')~\cite[Chapter~XIII]{Dgb}.

\begin{ques}[\bf increasing injections]
{\it Can one define without any set theoretical assumption a selfdistributive operation on increasing injections of~$\ZZZZp$?}
\end{ques}

The question seems very difficult. J.T.\,Moore wondered whether there could be a connection with the F{\o}lner function of Thompson's group~$F$~\cite{Moore}.

\subsection{Linear representations}\label{SS:Burau}

Linear representations of braid groups provide further quotients. The Lawrence--Krammer representation is known to be faithful~\cite{Big, Krb}, so its image is just isomorphic to its domain, and nothing new seems to be expectable here.

By contrast, the Burau representation is not faithful~\cite{LoP}, and, therefore, its image is a proper quotient. To work with~$\Bi$, we have to consider the direct limit $\mathrm{GL}_\infty(\ZZZZ[t, t\inv])$ of the groups~$\mathrm{GL}_\nn(\ZZZZ[t, t\inv])$ (identified with invertible $\nn$-mat\-rices) with the top--left embeddings of matrices. The (unreduced) version of the representation is defined by $$\rho(\sig\ii) = \Sigma_\ii := \sh^{\ii - 1}\bigg(\begin{pmatrix}1-t&t\\1&0\end{pmatrix}\bigg),$$
where $\sh$ is the obvious bottom--right shift of matrices. 

Projecting the braid operation~$\op$ yields a selfdistributive operation on the image of~$\rho$. The latter turns out to extend to arbitrary elements of~$\mathrm{GL}_\infty(\ZZZZ[t, t\inv])$:

\begin{prop}[\bf Burau shelf]
For~$\AA, \BB$ in~$\mathrm{GL}_\infty(\ZZZZ[t, t\inv])$, define
\begin{equation}\label{E:BurauOp}
\AA \op \BB:= \AA \cdot \sh(\BB) \cdot \Sigma_1 \cdot \sh(\AA)\inv.
\end{equation}
Then $(\mathrm{GL}_\infty(\ZZZZ[t, t\inv]), \op)$ is a left-shelf, and the Burau representation is a morphism from~$(\Bi, \op)$ to~$(\mathrm{GL}_\infty(\ZZZZ[t, t\inv]), \op)$.
\end{prop}

The verification is the same as for~$\Bi$: the point is that $\Sigma_1$ commutes with every matrix in the image of~$\sh^2$, and that $\Sigma_1 \Sigma_2 \Sigma_1 \Sigma_2\inv$ is equal to~$\Sigma_2 \Sigma_1$. We may wonder whether the identity matrix generates under~$\op$ a free left-shelf. The answer must be negative: if so, the division relation would have no cycle in~$(\mathrm{GL}_\infty(\ZZZZ[t, t\inv]), \op)$, implying that the Burau image of a $\sig1$-positive braid would never be trivial, implying in turn that $\rho$ is injective---which is false. By the way, an example of a $\sig1$-positive braid (with five~$\sig1$ and no~$\siginv1$) whose Burau image is trivial is constructed in~\cite{Dfh}. This corresponds to a cycle of length~$5$ for division in $(\mathrm{GL}_\infty(\ZZZZ[t, t\inv]), \op)$.

Of course, on the model of Question~\ref{Q:PermPres}, we may ask for a presentation of the subshelf of $(\mathrm{GL}_\infty(\ZZZZ[t, t\inv]), \op)$ generated by the identity matrix. But so little is known about the matrix operation~$\op$ that the question seems out from reach. Clearly, \eqref{E:BurauOp} implies the relation $\det(\AA \op \BB) = -t \cdot \det(\BB)$. A similar but more exotic relation is
$$\shtr(\AA \op \BB) = \shtr(\BB) + t,$$
where $\shtr(\AA)$ (``shifted trace'') is the sum of all overdiagonal entries~$\sum_\ii\AA_{\ii, \ii+ 1}$.

\subsection{Laver tables}\label{SS:Laver}

For every positive integer~$\NN$, there exists a unique binary operation~$\op$ on~$\{1 \wdots \NN\}$ that satisfies $\xx \op 1 = \xx + 1 \mod \NN$ and obeys the law $\xx \op (\yy \op 1) = (\xx \op \yy) \op (\xx \op 1)$; the structure so obtained is a left-shelf if, and only if, $\NN$ is a power of~$2$. The structure with~$2^\nn$~elements is called the \emph{$\nn$th Laver table}, usually denoted by~$\AA_\nn$, see Fig.~\ref{F:Laver} for the first four tables. Laver tables appear as the elementary building bricks for constructing all (finite) monogenerated shelves~\cite{Drd, Drf, Smed}, and can adequately be seen as counterparts of cyclic groups in the SD-world. We refer to~\cite[Chapter~X]{Dgd}, \cite[Chapitre~XIV]{Diy} (in French), or to~\cite{DraS} for a more complete introduction. For every~$\nn$, projection modulo~$2^\nn$ defines a surjective homomorphism from~$\AA_{\nn + 1}$ to~$\AA_\nn$. It is conjectured that the inverse limit of the system so obtained is free. A proof of the conjecture is known under some large cardinal axiom~\cite{Lvd}, a very unusual and puzzling situation.

\vspace{-3mm}

\begin{figure}[htb]
\small
\begin{equation*}
\begin{tabular}{c|p{0.5mm}}
$\AA_0$&$1$\\
\hline
$1$&$1$
\end{tabular}
\quad 
\begin{tabular}{c|p{0.5mm}p{0.5mm}}
$\AA_1$&$1$&$2$\\
\hline
$1$&$2$&$2$\\
$2$&$1$&$2$\\
\end{tabular}
\quad 
\begin{tabular}{c|p{0.5mm}p{0.5mm}p{0.5mm}p{0.5mm}}
$\AA_2$&$1$&$2$&$3$&$4$\\
\hline
$1$&$2$&$4$&$2$&$4$\\
$2$&$3$&$4$&$3$&$4$\\
$3$&$4$&$4$&$4$&$4$\\
$4$&$1$&$2$&$3$&$4$\\
\end{tabular}
\quad 
\begin{tabular}{c|p{0.5mm}p{0.5mm}p{0.5mm}p{0.5mm}p{0.5mm}p{0.5mm}p{0.5mm}p{0.5mm}}
$\AA_3$&$1$&$2$&$3$&$4$&$5$&$6$&$7$&$8$\\
\hline
$1$&$2$&$4$&$6$&$8$&$2$&$4$&$6$&$8$\\
$2$&$3$&$4$&$7$&$8$&$3$&$4$&$7$&$8$\\
$3$&$4$&$8$&$4$&$8$&$4$&$8$&$4$&$8$\\
$4$&$5$&$6$&$7$&$8$&$5$&$6$&$7$&$8$\\
$5$&$6$&$8$&$6$&$8$&$6$&$8$&$6$&$8$\\
$6$&$7$&$8$&$7$&$8$&$7$&$8$&$7$&$8$\\
$7$&$8$&$8$&$8$&$8$&$8$&$8$&$8$&$8$\\
$8$&$1$&$2$&$3$&$4$&$5$&$6$&$7$&$8$\\
\end{tabular}
\vspace{-3mm}
\end{equation*}
\caption{\sf The first four Laver tables; observe the periodic behaviour of the rows, which consist of the repetition of a pattern whose length is itself a power of~$2$, a general phenomenon.}
\label{F:Laver}
\end{figure}

Because of the fundamental position of Laver tables in the world of monogenerated left-shelves, a natural question is to connect them with the braid shelf~$(\Bi, \op)$. As $(\Bisp, \op)$ is a free monogenerated left-shelf, there must exist for every~$\nn$ a congruence~$\equiv_\nn$ on~$(\Bisp, \op)$ such that $\Bisp{/}{\equiv_\nn}$ is isomorphic to~$\AA_\nn$. 

\begin{ques}[\bf $\AA_\nn$ as quotient of~$\Bi$]\label{Q:Quot}
{\it Does there exist a simple topological/algeb\-raic/combinatorial characterization of the relation~$\equiv_\nn$ on special braids? Can this relation be extended to arbitrary braids in a natural way?}
\end{ques}

At the moment, almost nothing is known about the question. In lack of an answer, we can look at the possible connection between the known quotient of~$(\Bi, \op)$, namely the left-shelf~$(\Si, \op)$ of Section~\ref{SS:Perm}, and Laver tables. Such a connection exists typically for the $2$-element Laver table~$\AA_1$.

\begin{prop}[\bf $\AA_1$ as quotient of~$\Si$]\label{P:Class}
Define the \emph{class}~$\cl\ff$ of a permutation~$\ff$ in~$\Si$ to be~$\ff\inv(1)$. Let~$\Sii:= \{\ff \in \Si \mid \cl\ff \le 2\}$ $($``small class'' permutations$)$.

\ITEM1 The set~$\Sii$ is closed under~$\op$, and it includes~$\Sisp$.

\ITEM2 Class equality is a congruence on the left-shelf~$(\Sii, \op)$, hence on~$(\Sisp, \op)$. In both cases, the quotient is the Laver table~$\AA_1$.
\end{prop}

\begin{proof}
\ITEM1 As $\ss_{\nn - 1} {\pdots} \ss_2 \ss_1$ maps~$1$ to~$\nn$, a permutation~$\ff$ in~$\Si$ is of class~$\nn$ if, and only if, $\ff \ss_{\nn - 1} {\pdots} \ss_1$ is of class~$1$. On the other hand, a permutation~$\ff$ in~$\Si$ keeps~$1$ fixed if, and only if, it lies in the image of the shift mapping. Thus, for every~$\nn$, a permutation~$\ff$ has class~$\nn$ if, and only if, it can be written as $\sh(\ff') \ss_1 \ss_2 {\pdots} \ss_{\nn - 1}$ for some~$\ff'$. In particular, the elements of~$\Sii$ are of the form~$\sh(\ff')$ or~$\sh(\ff') \ss_1$.

Assume that $\ff$ has class~$1$, say $\ff = \sh(\ff')$. For every~$\gg$ in~$\Si$, we find
\begin{equation*}
\ff \op \gg = \sh(\ff') \, \sh(\gg) \, \ss_1 \, \sh^2(\ff')\inv = \sh(\ff') \, \sh(\gg) \, \sh^2(\ff')\inv \, \ss_1,
\end{equation*}
of class~$2$.
Assume now that $\ff$ has class~$2$, say $\ff = \sh(\ff') \ss_1$, and~$\gg$ has class~$1$, say~$\gg = \sh(\gg')$. Using again that $\ss_1$ commutes with~$\sh^2(\ff')$ and $\ss_1^2 = 1$, we find
\begin{equation*}
\ff \op \gg = \sh(\ff') \, \ss_1 \, \sh^2(\gg') \, \ss_1 \ss_2  \, \sh^2(\ff')\inv = \sh(\ff') \, \sh^2(\gg') \, \ss_2  \, \sh^2(\ff')\inv,
\end{equation*}
explicitly of class~$1$. Finally, assume that $\ff$ has class~$2$, say $\ff = \sh(\ff') \ss_1$, and~$\gg$ has class~$2$, say~$\gg = \sh(\gg') \ss_1$. Using now $\ss_1\ss_2\ss_1\ss_2 = \ss_2\ss_1$, we find
\begin{equation*}
\ff \op \gg = \sh(\ff') \, \ss_1 \, \sh^2(\gg') \, \ss_2\ss_1 \ss_2 \, \sh^2(\ff')\inv = \sh(\ff') \, \sh^2(\gg') \, \ss_2  \, \sh^2(\ff')\inv \, \ss_1,
\end{equation*}
of class~$1$. Thus, applying~$\op$ to permutations in~$\Sii$ yields permutations in~$\Sii$, that is, $\Sii$ is closed under~$\op$. Because $\id$ belongs to~$\Sii$ (it has class~$1$), the substructure of~$(\Si, \op)$ generated by~$\id$, which is~$\Sisp$, is included in~$\Sii$.

\ITEM2 Let $\equiv$ denote class equality on~$\Sii$. The above three computations show that $\equiv$ is compatible with the operation~$\op$, and that the table of~$\Sii{/}{\equiv}$ coincides with that of the Laver table~$\AA_1$, as displayed in Fig.~\ref{F:Laver}.
\end{proof}

Composing with the projection of braids to permutations, we deduce a partial answer to Question~\ref{Q:Quot} in the case of~$\AA_1$: 

\begin{coro}
Call a braid~$\br$ of class~$\nn$ if, in any braid diagram representing~$\br$, the strand starting at position~$1$ finishes at position~$\nn$. Let~$\Bii:= \{\br \in \Bi \mid \cl\br \le\nobreak 2\}$. Then $\Bii$ is closed under~$\op$ and includes~$\Bisp$. Class equality is a congruence on~$(\Bii, \op)$ and~$(\Bisp, \op)$, and the quotient is the Laver table~$\AA_1$.
\end{coro}

This answer is not fully satisfactory, in that the considered equivalence relation is defined only on a proper subset of~$\Bi$, namely~$\Bii$. However, the situation with~$\AA_2$ is worse:

\begin{fact}
{\it There is no congruence on~$(\Sisp, \op)$ such that the quotient is isomorphic to a Laver table~$\AA_\nn$ with~$\nn \ge 2$.} 
\end{fact}

\begin{proof}
As $\AA_2$ is a quotient of every table~$\AA_\nn$ with $\nn \ge 2$, it suffices to consider~$\AA_2$. Assume that $\phi$ is a morphism from~$(\Sisp, \op)$ to~$\AA_2$. As $\id$ is the unique generator of~$(\Sisp, \op)$ and $1$ is the unique generator of~$\AA_2$, we necessarily have $\phi(\id) = 1$. Now, we see on Fig.~Ê\ref{F:Laver} that, in~$\AA_2$, we have $1 \op 3 = 2$ and $3 \op 2 = 4$, hence
\begin{equation}
1 \op 1_{[3]} \not= 1_{[3]} \op 1_{[2]},
\end{equation}
whereas, in~$\Sisp$, we read in Fig.~\ref{F:Perm}
\begin{equation}\label{E:PermObst}
\id \op \id_{[3]} = \ss_3 \ss_1 = \id_{[3]} \op \id_{[2]}:
\end{equation}
this contradicts the assumption that $\phi$ is a homomorphism.
\end{proof}

The above fact does not prove that there is no answer to Question~\ref{Q:Quot}, but it shows that, for~$\nn \ge 2$, the relation~$\equiv_\nn$ cannot be defined in terms of the associated permutations. It might be interesting to try further quotients of the braid groups~\cite{Cox}, or linear representations. Typically, considering the Burau representation of Section~\ref{SS:Burau} leads to

\begin{ques}[\bf Burau]
Is the Laver table~$\AA_2$ a quotient of $(\mathrm{GL}_\infty(\ZZZZ[t, t\inv]), \op)$?
\end{ques}

At the least, writing~$\II$ for the identity matrix, one checks $\II \op \II_{[3]} \not= \II_{[3]} \op \II_{[2]}$, so the obstruction~\eqref{E:PermObst} in~$\Si$ vanishes in~$\mathrm{GL}_\infty(\ZZZZ[t, t\inv])$.

\section{Extensions}\label{S:Ext}

We conclude with hints about extensions of the braid operation~$\op$ to larger structures. In Subsection~\ref{SS:Charged}, we describe the group~$\CBi$ of charged braids, a natural extension of~$\Bi$, in which extra space enables one to realize free shelves on any number of generators. Next, in Subsection~\ref{SS:Ext}, we describe the monoid~$\EBi$ of extended braids, in which a second, associative operation exists beside the selfdistributive one. Finally, we mention in Subsection~\ref{SS:Par} the group of parenthesized braids~$\PBi$, also called the braided Thompson group~$\widehat{BV}$, another extension of~$\Bi$ equipped with a selfdistributive operation.

\subsection{Charged braids}\label{SS:Charged}

By Prop.~\ref{P:Free}, every braid generates under the selfdistributive operation~$\op$ a substructure that is a free monogenerated left-shelf. Plenty of space is left aside: for instance, we saw that no generator~$\sig\ii$ with $\ii \ge 2$ belongs to the substructure~$\Bisp$ generated by~$1$.

\begin{ques}[\bf free family]\label{Q:FreeFam}
{\it Does $(\Bi, \op)$ include a free left-shelf of rank~$2$, that is, does there exists a cardinal~$2$ free family in~$(\Bi, \op)$ (in the sense of Def.~\ref{D:Free})?}
\end{ques}

A negative answer is likely. Indeed, one can show~\cite{Dgb} that, for all~$\br_1, \br_2, \br$ in~$\Bi$ such that $\br$ is special, $\br_1 \op \br^{[\mm]} = \br_2 \op \br^{[\mm]}$ holds for~$\mm$ large enough, whereas, if $\{\br_1, \br_2\}$ is free in~$(\Bi, \op)$, then $\br_1 \op \br = \br_2 \op \br$ holds for no~$\br$ in~$\langle\br_1, \br_2\rangle$. Therefore, if $\{\br_1, \br_2\}$ is a free family, the subshelf $\langle\br_1, \br_2\rangle$ generated by~$\br_1$ and~$\br_2$ contains~no special braid. So, in particular, there exists no braid~$\br$ such that $\{1, \br\}$ is free.

In order to possibly construct a braid realization for free left-shelves of rank larger than~$1$, it is therefore natural to introduce extensions of~$\Bi$ in which enough space is explicitly granted. A first solution was described by D.\,Larue in~\cite{Lrb} using a (very large) algebraic extension. Here we briefly mention the alternative solution of~\cite{Dfn}, which is simpler and admits a natural topological description.

\begin{defi}[\bf charged braids]
For~$\nn \ge 1$, we let~$\CBR\nn$ be the extension of~$\BR\nn$ obtained by adding mutually commuting elements~$\rho_1 \wdots \rho_\nn$ subject to the relations
\begin{equation}
\sig\ii \rho_\jj = \rho_\jj \sig\ii \text{\quad for $\jj < \ii$ or $\jj \ge \ii + 2$}, \qquad \sig\ii \rho_\ii \rho_{\ii + 1} = \rho_\ii \rho_{\ii + 1} \sig\ii.
\end{equation}
\end{defi}

In topological terms, $\CBR\nn$ is the group of isotopy classes of enhanced braid diagrams, in which the strands wear integer-valued charges and $\rho_\ii$ corresponds to adding an elementary charge~$+1$ on the $\ii$th strand, see Fig.~\ref{F:Charged}. The rule is that charges freely move on the strands, but a charge on the $\ii$th strand may move through a crossing~$\sigg\ii{\pm1}$ if, and only if, a similar charge on the $(\ii + 1)$st strand simultaneously does, see Fig.~\ref{F:Charged}. 

\begin{figure}[htb]
$$\begin{picture}(42,16)(0,0)
\psline[linewidth=2pt]{c->}(0,16)(28,16)(36,8)(42,8)
\psline[linewidth=2pt]{c-}(0,0)(4,0)(12,8)(16,8)
\psline[linewidth=2pt,border=4pt]{c->}(0,8)(4,8)(12,0)(16,0)(24,8)(28,8)(36,16)(42,16)
\psline[linewidth=2pt,border=4pt]{c->}(16,8)(24,0)(42,0)
\put(1,9.5){$\scriptstyle\oplus$}
\put(0,1.5){$\scriptstyle\oplus\oplus$}
\put(12,9.5){$\scriptstyle\ominus\ominus$}
\put(24,17.5){$\scriptstyle\oplus\oplus$}
\put(36,9.5){$\scriptstyle\ominus\ominus$}
\put(37,17.5){$\scriptstyle\ominus$}
\put(48.5,7.5){$\sim$}
\end{picture}
\hspace{16mm}
\begin{picture}(36,16)(0,0)
\psline[linewidth=2pt]{c->}(0,16)(23,16)(30,8)(34,8)
\psline[linewidth=2pt](0,0)(2,0)(10,8)(12,8)
\psline[linewidth=2pt,border=4pt]{c->}(0,8)(2,8)(10,0)(12,0)(20,8)(22,8)(30,16)(34,16)
\psline[linewidth=2pt,border=4pt]{c->}(12,8)(20,0)(34,0)
\put(0,1.5){$\scriptstyle\oplus$}
\put(10,1.5){$\scriptstyle\oplus$}
\put(10,9.5){$\scriptstyle\ominus$}
\put(20,17.5){$\scriptstyle\oplus$}
\put(20,9.5){$\scriptstyle\ominus$}
\put(30,9.5){$\scriptstyle\ominus$}
\end{picture}
$$
\caption{\sf Typical charged braid diagrams, represented with~$\oplus$ on the $\ii$th strand for~$\rho_\ii$, and $\ominus$ for~$\rho_\ii\inv$: here the diagrams encoded by $\rho_1^2 \rho_2 \sig1 \rho_2^{-2} \sig1 \rho_3^2 \siginv2 \rho_3\inv \rho_2^{-2}$, corresponding to evaluating~$(\rho_1^2 \op \rho_1) \op 1$ with the operation of  Prop.~\ref{P:Charged}, and the equivalent diagram $\rho_1 \sig1 \rho_1 \rho_2\inv \sig1 \rho_3 \rho_2\inv \siginv2 \rho_2\inv$, in which two initial $\oplus$ charges have been moved rightwards through~$\sig1$, one cancelling one of the~$\ominus$, and, similarly, two final $\ominus$ charges have been moved leftwards through~$\siginv2$, one cancelling one of the~$\oplus$.}
\label{F:Charged}
\end{figure}

The group~$\BR\nn$ embeds in~$\CBR\nn$. On the other hand, erasing charges, that is, collapsing all generators~$\rho_\ii$ defines a projection~$\pi_\nn$ from~$\CBR\nn$ to~$\BR\nn$, which is a retraction of the embedding. Note that $\Ker(\pi_\nn)$ is, already for $\nn = 2$, quite big: mapping an even length sequence of integers $(\ee_0 \wdots \ee_{2\ell}, \ee)$ to 
$$\rho_1^{\ee_0} \sig1 \rho_1^{\ee_1} \siginv1 \rho_1^{\ee_2} \sig1 \rho_1^{\ee_3} \siginv1 \pdots \siginv1 \rho_1^{\ee_{2\ell}} \rho_2^\ee$$
produces pairwise distinct elements of~$\Ker(\pi_2)$ since the only eligible relations consist in moving the final~$\rho_2^\ee$ and shifting some~$\rho_1^{\ee_\ii}$ accordingly.

We denote by~$\CBi$ the direct limit of the groups~$\CBR\nn$ when $\nn$ grows to infinity.

\begin{prop}[\bf charged braid shelf]\label{P:Charged}
Extend the operation~$\op$ from~$\Bi$ to~$\CBi$ by defining $\sh(\rho_\ii) := \rho_{\ii + 1}$ for every~$\ii$ and keeping~\eqref{E:BraidOp}. Then $(\CBi, \op)$ is a left-shelf, and, for every~$\nn \ge 2$, the family~$\{1, \rho_1 \wdots \rho_1^{\nn - 1}\}$ is free in~$(\CBi, \op)$.
\end{prop}

The proof relies on extending the freeness criterion of Lemma~\ref{L:Freeness} to rank~$\ge 2$, which is not very hard as, essentially, free shelves of higher rank are sort of lexicographic extensions of rank~$1$ free shelves~\cite[Prop.~V.6.6]{Dgd}.

Let us mention that nothing is known about the following informal question, related to Question~\ref{Q:FreeFam}:

\begin{ques}[\bf generators]\label{Q:Gen}
Does the structure~$(\Bi, \op)$ admit a natural family of generators?
\end{ques}

\subsection{Extended braids}\label{SS:Ext}

Another extension of potential interest in topology involves the monoid of ``extended braids''~\cite{Dgb}. In a number of (left) shelves~$(\SS, \op)$, there exists a second, associative operation~$\comp$, admitting a unit~Ê$1$ and connected with~$\op$ by the mixed laws
\begin{gather}\label{E:Second1}
(\xx \comp \yy) \op \zz = \xx \op (\yy \op \zz), \qquad \xx \op (\yy \comp \zz) = (\xx \op \yy) \comp (\xx \op \zz)\\ \label{E:Second2}
\xx \comp \yy = (\xx{\,\op\,}\yy) \comp \xx,  \qquad 1 \op \xx = \xx, \qquad \xx \op 1 = 1.
\end{gather}
In particular, if $\GG$ is a group and $\op$ is the conjugation defined by $\xx \op \yy:= \xx\yy\xx\inv$, then the multiplication of~$\GG$ provides such a second operation. 

Whereas the selfdistributive operation can be used to label the strands of a (braid or knot) diagram, such a second operation can be used to label the regions between the strands using the convention that crossing from left to right a strand labelled~$\aa$ multiplies the region label by~$\aa$. Then the first law of~\eqref{E:Second2} expresses the coherence of region and strand colourings, according to the scheme
\begin{equation}\label{E:ColRegion}
\VR(13,11)\begin{picture}(21,0)(-5,7)
\psline[linewidth=2pt]{c->}(-6,1)(1,1)(16,16)(25,16)
\psline[linewidth=2pt,border=4pt]{c->}(-6,16)(1,16)(16,1)(25,1)
\put(-5,2.5){$\aa$}
\put(20,17.5){$\aa$}
\put(-5,17.5){$\bb$}
\put(19,2.5){$\aa\op\bb$}
\psline[linewidth=0.5pt,arrowsize=3pt]{->}(5,15)(1,11)\put(7,16.5){$\xx$}
\psline[linewidth=0.5pt,arrowsize=3pt]{->}(11,15)(15,11) \put(-7,7.5){$\bb {\,\comp\,} \xx$}
\psline[linewidth=0.5pt,arrowsize=3pt]{->}(15,5)(11,1)\put(16,7.5){$\aa {\,\comp\,} \xx$}
\psline[linewidth=0.5pt,arrowsize=3pt]{->}(1,5)(5,1) \put(-8,-3){$\aa {\,\comp\,}\bb {\,\comp\,} \xx =   (\aa\op\bb) {\,\comp\,} \aa {\,\comp\,} \xx$}
\end{picture}
\end{equation}
It is worth noting that the first law of~\eqref{E:Second2} and the associativity of~$\comp$ imply that the map $(\aa, \xx) \mapsto \aa \comp \xx$ turns~$\SS$ into an $\SS$-module (or $\SS$-set), so the region labelling of~\eqref{E:ColRegion} is a particular case of the shadow colorings from knot theory~\cite{Kam}. The second operation~$\comp$ can also be used to color braids with zip and unzip vertices; in this case all the mixed laws for~$\comp$ and~$\op$ receive a topological meaning, appearing as algebraic distillations of R-moves for knotted graphs~\cite{BBF}.

It is therefore natural to look for a possible second operation on~$(\Bi, \op)$. First, the operation~$\comp$ cannot be the group multiplication, nor any group operation on~$\Bi$, since \eqref{E:Second2} would then imply $\xx \op \yy = \xx \comp \yy \comp \xx\inv$, which contradicts $\op$ not being idempotent. Hence, the best we could expect is a monoid structure that properly extends the group structure of~$\Bi$. 

\noindent\begin{minipage}{\textwidth}\rightskip40mm
\hspace{19pt}\VR(3.5,0)Such a structure exists, and it can be defined as a (partial) topological completion of~$\Bi$~\cite{Dgb}. Hereafter, we denote by~$\tau_{\pp, \nn}$ the positive braid where $\nn$~strands cross over $\pp$~strands. So $\tau_{\pp, 0}$ is~Ê$1$, and $\tau_{2,3}$ is shown on the right.\hfill
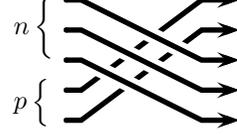
\begin{picture}(0,0)(-14,0)
\psline[linewidth=2pt,border=2pt]{c->}(0,0)(2,0)(18,12)(23,12)
\psline[linewidth=2pt,border=2pt]{c->}(0,4)(2,4)(18,16)(23,16)
\psline[linewidth=2pt,border=2pt]{c->}(0,8)(2,8)(18,0)(23,0)
\psline[linewidth=2pt,border=2pt]{c->}(0,12)(2,12)(18,4)(23,4)
\psline[linewidth=2pt,border=2pt]{c->}(0,16)(2,16)(18,8)(23,8)
\put(-7,1.5){$\pp\left\{\vbox to 4mm{\vfill}\right.$}
\put(-7,11.5){$\nn\left\{\vbox to 5mm{\vfill}\right.$}
\end{picture}
\end{minipage}

\begin{lemm}[\bf topology]\label{L:Dist}
For~$\br_1, \br_2$ in~$\Bi$, put $\dist(\br_1, \br_2):= 0$ for $\br_1 = \br_2$ and $\dist(\br_1, \br_2):=\nobreak 2^{-\pp}$ for $\br_1\inv \br_2$ in~$\Im(\sh^\pp) \setminus\nobreak \Im(\sh^{\pp - 1})$. Then $\dist$ is an ultrametric distance on~$\Bi$. For every braid~$\br$ and every~$\pp$, the sequence $(\br\tau_{\pp, \nn})_{\nn \ge 0}$ is a Cauchy sequence without limit in~$\Bi$. Furthermore, the limits of~$(\br\tau_{\pp, \nn})_{\nn \ge 0}$ and ~$(\brr\tau_{\qq, \nn})_{\nn \ge 0}$ (in a completion) coincide if, and only if, we have $\pp = \qq$ and $\br\inv \brr \in \BR\pp$.
\end{lemm}

The key point in the proof is the conjugation formula $\tau_{\pp, \nn}\inv \br \tau_{\pp, \nn} = \sh^\nn(\br)$, which holds for every~$\nn$ and every~$\br$ in~$\BR\pp$, as a picture similar to Fig.\ref{F:Trick} shows. Using the last statement in Lemma~\ref{L:Dist}, it is easy to describe a completion:

\begin{prop}[\bf completion]
For $(\br, \pp), (\brr, \qq)$ in~$\Bi \times \NNNN$, declare $(\br, \pp) \equiv\nobreak (\brr, \qq)$ for $\pp = \qq$ and $\br\inv \brr \in \BR\pp$, write $[\br, \pp]$ for the $\equiv$-class of~$(\br, \qq)$, put $\dist([\br, \pp], [\brr, \qq]):= \lim_{\nn\to\infty}\dist(\br \tau_{\pp, \nn}, \brr\tau_{\qq, \nn})$, and define~$\EBi$ to be the family of all $\equiv$-classes. Then mapping~$\br$ to~$[\br, 0]$ identifies~$\Bi$ with a dense subset of~$\EBi$, and the latter is obtained by adding~$[\br, \pp]$ as the limit of $(\br\tau_{\pp, \nn})_{\nn \ge 0}$ for all~$\br$ and~$\pp$. 
\end{prop}

\noindent\begin{minipage}{\textwidth}\rightskip40mm
\hspace{19pt}\VR(3,0)The elements of~$\EBi$ are called ``\emph{extended braids}''. The introduction of~$[\br, \pp]$ as the limit of the braids~$\br\tau_{\pp, \nn}$ when $\nn$ goes to~$\infty$ makes it natural to associate with the extended braid~$[\br, \pp]$ the diagram shown on the right, where $\pp$~strands vanish at infinity.\hfill
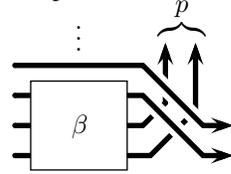
\begin{picture}(0,0)(-8,-2)
\psline[linewidth=2pt,border=2pt]{c->}(0,0)(18,0)(24,6)(24,15)
\psline[linewidth=2pt,border=2pt]{c->}(0,4)(17,4)(20,7)(20,15)
\psline[linewidth=2pt,border=2pt]{c->}(0,8)(17,8)(25,0)(29,0)
\psline[linewidth=2pt,border=2pt]{c->}(0,12)(17,12)(25,4)(29,4)
\put(19,15.5){$\overbrace{\hbox to 3mm{\hfill}}^{\textstyle\pp}$}
\psframe[linewidth=0.5pt,fillstyle=solid,fillcolor=white](2,-2)(15,10)
\put(7.5,3){$\br$}
\put(8,14){$\vdots$}
\end{picture}
\end{minipage}

\VR(3.5,0)Now the nice feature is that $\EBi$ possesses a rich algebraic structure:

\begin{prop}[\bf operations]
For $[\br, \pp], [\brr, \qq]$ in~$\EBi$, define
\begin{gather}
[\br, \pp] \cdot [\brr, \qq]:= [\br \cdot \sh^\pp(\brr), \pp + \qq],\\
\label{E:ExtOp2}
[\br, \pp] \, \oph\,  [\brr, \qq]:= [\br \cdot \sh^\pp(\brr) \cdot \tau_{\pp, \qq} \cdot \sh^\qq(\br)\inv, \qq].
\end{gather}
\ITEM1 $(\EBi, \cdot, 1)$ is a monoid; $\Bi$ is its group of units, and $\EBi$ is obtained from~$\Bi$ by adding a unique additional element, namely $\tau:= [1, 1]$, subject to the relations
$$\sig1 \tau^2 = \tau^2 \quad \text{and} \quad \sig\ii \tau = \tau \sig{\ii - 1} \text{\quad for $\ii \ge 2$}.$$
\ITEM2 $(\EBi, \oph)$ is a left-shelf; $(\EBi, \oph, \cdot)$ obeys the mixed laws~\eqref{E:Second1} and \eqref{E:Second2}.
\end{prop}

Point~\ITEM1 implies the equality $[\br, \pp] = \br \, \tau^\pp$ for all~$\br$ and~$\pp$. By construction, $\EBi$ is the disjoint union of the groups~$\Bi{/}\BR\pp$. As $\BR0$ and $\BR1$ are trivial, $\EBi$ includes two copies of~$\Bi$, namely~$\Bi$, and $\Bi\tau:= \{\br \tau \mid \br \in\nobreak \Bi\}$. Then the selfdistributive operation~$\oph$ leaves each layer~$\Bi{/}\BR\pp$ stable. On~$\Bi$, \eqref{E:ExtOp2} cooks down to group conjugacy: $\br \oph \brr = \br \brr \br\inv$. By constrast, on the second copy~$\Bi\tau$, \eqref{E:ExtOp2} yields $\br\tau \oph \brr\tau = \br \, \sh(\brr) \, \sig1 \, \sh(\br)\inv \, \tau$, in which we recognize the operation~$\op$ of~\eqref{E:BraidOp} on~$\Bi$. 

So, at the expense of embedding~$\Bi$ into a (moderately) larger space, we obtained a complete positive solution to the initial two operations problem.

\subsection{Parenthesized braids}\label{SS:Par}

We conclude with still another extension of the braid group~$\Bi$ that contains an interesting selfdistributive structure, namely the braided Thompson group~$\BVhat$ of~\cite{Bri1, Bri2}, also described as the group of parenthesized braids and denoted~$\PBi$ in~\cite{Dhe}.

\begin{prop}[\bf parenthesized braids]\label{P:ParBraid}
Let~$\PBi$ be the extension of~$\Bi$ obtained by adding an infinite sequence of generators~$\aaa1, \aaa2$,... subject to
\begin{gather} \label{E:Relations}
\sig{i+1} \sig i \aaa{i+1} = \aaa i \sig i \text{\quad and \quad} \sig i \sig{i+1} \aaa i = \aaa{i+1} \sig i \text{\quad for every~$\ii$},\\
\sig i \aaa j  = \aaa j \sig i, \quad \aaa i \aaa{j-1} = \aaa j \aaa i, \text{\quad and \quad}  \aaa i \sig{j-1} = \sig j \aaa i \text{\quad for $\jj \ge \ii +2$}.
\end{gather}
Extending~$\sh$ by $\sh(\aaa\ii) := \aaa{\ii + 1}$ for every~Ê$\ii$, define~$\op$ and~$\comp$ on~$\PBi$ by
\begin{equation}
\br \op \brr:= \br \cdot \sh(\brr) \cdot \sig1 \cdot \sh(\br)\inv \text{\quad and \quad} \br \comp \brr:= \br \cdot \sh(\brr) \cdot \aaa1.
\end{equation}
Then $(\PBi, \op)$ is a left-shelf, and $(\PBi, \op, \comp)$ obeys the mixed laws~\eqref{E:Second1}.
\end{prop}

In the (large) group~$\PBi$, the elements~$\sig\ii$ generate a copy of~$\Bi$, whereas the elements~$\aaa\ii$ generate a copy of R.\,Thompson's group~$F$~\cite{CFP}. Topologically, the elements of~$\PBi$ can be thought of as braids in which the distance between adjacent strands need not be constant. Then $\aaa\ii$ corresponds to moving the strand(s) at position~$\ii +\nobreak 1$ or infinitely close to position~$\ii + \varepsilon$ with $\varepsilon$ infinitely small (but not merging them: a subsequent~$\aaa\ii\inv$ may separate them back). It is shown in~\cite{Dhe} and in~\cite{Bri2} that $\PBi$ is a group of right fractions for the submonoid~$\PBi^+$ generated by the $\sig\ii$s and the $\aaa\jj$s, and that the latter is a Zappa--Szep product of the monoids~$\Bi^+$ and~$F^+$. This implies that every element of~$\PBi$ can be expressed as~$\ff\inv \br \gg$, with $\br$ in~$\Bi$ and $\ff, \gg$ in~$F^+$, resulting in diagrams like the one shown in Fig.~\ref{F:Par}.

\begin{figure}[htb]
$$\begin{picture}(40,16)(0,0)
\psline[linewidth=2pt]{c-}(0,8)(8,16)(16,8)
\psline[linewidth=2pt]{c-}(0,0)(16,0)(32,16)
\psline[linewidth=2pt,border=3pt]{c->}(0,1)(8,8)(16,16)(24,16)(32,8)(40,8)
\psline[linewidth=2pt,border=3pt]{->}(16,8)(24,0)(40,0)
\psline[linewidth=2pt](0,0)(16,0)
\psline[linewidth=2pt]{->}(32,16)(40,9)
\end{picture}$$
\caption{\sf Typical parenthesized braid diagram witnessing for the decomposition of~$\PBi^+$ as a Zappa--Szep product of~$\Bi^+$ and~$F^+$, here $\aaa1\inv \siginv2 \sig1 \sig2 \aaa2$: diverging, braiding, and finally merging.}
\label{F:Par}
\end{figure}

The selfdistributive structure on~$\PBi$ is not so perfect as the one on~$\EBi$ in that the second operation is not associative, and the laws of~\eqref{E:Second2} all fail in~$(\PBi, \op, \comp)$. However, calling~$\PBisp$ the closure of~$1$ under~$\op$ and~$\comp$ in~$\PBi$, one obtains decompositions of arbitrary parenthesized braids into special ones as in~Prop.~\ref{P:Decomp}. From there, one can order~$\PBi$, and, calling ``augmented left-shelf'' a left-shelf equipped with a second operation satisfying~\eqref{E:Second1}, establish that $\PBisp$ is a free monogenerated augmented left-shelf~\cite{Dhj, Dhl}.


\end{document}